\documentclass[a4paper,reqno]{amsart}
\usepackage{amsfonts}
\usepackage{amssymb}
\usepackage{enumitem}
\usepackage{color}
\usepackage[colorlinks=true]{hyperref}
\hypersetup{urlcolor=cyan,citecolor=cyan}

\setlist[enumerate]{wide,label=\emph{(\roman*)}}

\usepackage{environ}

\newtheorem{theorem}{Theorem}[section]

\newtheorem{lemma}[theorem]{Lemma}
\newtheorem{proposition}[theorem]{Proposition}

\theoremstyle{definition}
\newtheorem{definition}[theorem]{Definition}
\newtheorem{remark}[theorem]{Remark}
\numberwithin{equation}{section}

\newcommand{\R}{\mathbb{R}}

\def\bell{\boldsymbol \ell}
\def\bb{\boldsymbol{b}}
\def\be{{\boldsymbol{\rm{e}}}}
\def\E{\dot{H}^{1}\times L^{2}}
\def\pshbb1#1#2{\left(#1,#2\right)_{\dot H^1_{{{\ell}}_1}}}

\def\bell{\boldsymbol{\ell}}

\def\RR{\mathbb{R}^{5}}
\def\d{{\rm d}}
\def\E{\mathcal{H}}

\DeclareMathOperator\SPAN{Span}
\DeclareMathOperator\CARD{Card}

\begin{document}

\parindent=0pt

	\title[multi-solitons for 5D nonlinear wave equation]{Construction of excited multi-solitons for the 5D energy-critical wave equation}
	\author{XU YUAN}
	
	\address{CMLS, \'Ecole polytechnique, CNRS, Institut Polytechnique de Paris, F-91128 Palaiseau Cedex, France.}
	
	\email{xu.yuan@polytechnique.edu}
	\begin{abstract}
	For the 5D energy-critical wave equation, we construct excited $N$-solitons with collinear speeds, \emph{i.e.} solutions $u$ of the equation such that
	\begin{equation*}
	\lim_{t\to+\infty}\bigg\|\nabla_{t,x}u(t)-\nabla_{t,x}\bigg(\sum_{n=1}^{N}Q_{n}(t)\bigg)\bigg\|_{L^{2}}=0,
	\end{equation*}
	where for $n=1,\ldots,N$, $Q_n(t,x)$ is the Lorentz transform of a non-degenerate and sufficiently decaying excited state, each with different but collinear speeds. The existence proof follows the ideas of Martel-Merle~\cite{MMwave1} and
	C\^ote-Martel~\cite{CMakg} developed for the energy-critical wave and nonlinear Klein-Gordon equations. In particular, we
	rely on an energy method and on a general coercivity property for the linearized operator.
	\end{abstract}
	\maketitle
\section{Introduction}
\subsection{Main result}
We consider the energy-critical focusing wave equation in dimension 5,
\begin{equation}\label{wave}
\left\{ \begin{aligned}
&\partial_t^2 u - \Delta u - |u|^{\frac{4}{3}}u = 0, \quad (t,x)\in [0,\infty)\times \mathbb{R}^5,\\
& u_{|t=0} = u_0\in \dot H^1,\quad 
\partial_t u_{|t=0} = u_1\in L^2.
\end{aligned}\right.
\end{equation}
Recall that the Cauchy problem for equation~\eqref{wave} is locally well-posed in the energy space $ {\dot{H}}^{1}\times L^{2}$. See $\emph{e.g.}$~\cite{KM} and references therein. Let $f(u)=|u|^{\frac{4}{3}}u$ and $F(u)=\frac{3}{10}|u|^{\frac{10}{3}}$. For any $\dot{H}^{1}\times L^{2}$ solution $\vec{u}=(u,\partial_{t}u)$, the energy $E$ and the momentum $P$ are conserved along the flow, where
\begin{equation*}
E(u,\partial_{t} u)=\frac{1}{2}\int_{\mathbb{R}^{5}}\bigg(|\nabla u|^{2}+|\partial_{t} u|^{2}-2F(u)\bigg)\d x,\quad P(u,\partial_{t} u)=\frac{1}{2}\int_{\mathbb{R}^{5}}\big(\partial_{t}u \nabla u\big) \d x.
\end{equation*}
Looking for stationnary solutions $u(t,x)=q(x)$ of~\eqref{wave} in $\dot{H}^{1}$, 
we reduce to the Yamabe equation
\begin{equation}\label{Yamabe}
\Delta q+f(q)=0\quad \mbox{in}\ \RR.
\end{equation}
Denote
\begin{equation*}
\Sigma=\left\{q\in \dot{H}^{1}(\RR)/0: q\ \mbox{satisfies}~\eqref{Yamabe}\right\}.
\end{equation*}
For $q\in \Sigma$ and $\bell\in \R^{5}$ such that $|\bell|<1$, let
\begin{equation}\label{defQl}
q_{\bell}(x)=q\left(\frac{1}{|\bell|^{2}}\left(\frac{1}{\sqrt{1-|\bell|^{2}}}\bell\cdot x\right)\bell+x\right),
\end{equation}
then $u(t,x)=q_{\bell}(x-\bell t)$ is a global, bounded solution of~\eqref{wave}.

It is known (see \emph{e.g.}~\cite{Aub,Talen}) that
the unique (up to scaling invariance and sign change) \emph{radially symmetric} element of $\Sigma$ is the \emph{ground state} $W$ given explicitly by 
\begin{equation*}
W(x)=\left(1+\frac{|x|^{2}}{15}\right)^{-\frac{3}{2}}.
\end{equation*}
The existence of non-radially symmetric, sign changing elements of $\Sigma$ with arbitrary large energy
was first proved by Ding~\cite{Ding}, using variational arguments. Functions $q\in \Sigma$ with $q\ne W$ 
(up to invariances) are usually called \emph{excited states}.

For any $q\in \Sigma$, we denote the linearized operator around $q$ by
\begin{equation*}
\mathcal{L}=-\Delta-f'(q).
\end{equation*}
Set 
\begin{equation*}
\mathcal{Z}_{q}=\left\{f\in \dot{H}^{1}\ \mbox{such that}\ \mathcal{L}f=0\right\},
\end{equation*}
and
\begin{multline*}
\widetilde{{\mathcal{Z}}}_{q}=
 \SPAN \bigg\{ \frac{3}{2}q+x\cdot \nabla q;
(x_{i}\partial_{x_{j}}q-x_{j}\partial_{x_{i}}q),1\le i<j\le 5;\\
 \partial_{x_{i}}q ;  -3x_{i}q+|x|^{2}\partial_{x_{i}}q-2x_{i} (x\cdot \nabla q),1\le i\le 5\bigg\}.
\end{multline*}
The function space $\widetilde{\mathcal{Z}}_{q}\subset \mathcal{Z}_{q}$ is the null space of $\mathcal{L}$ that is generated by a family of explicit transformations (see~\cite[Lemma 3.8]{DKMASNSP} and \S\ref{SS:trans}).

We will need the following definitions.
\begin{definition}
	Let $q\in \Sigma$.
	\begin{enumerate}
		\item[$\rm{(i)}$] $q$ is called a \emph{non-degenerate state} if $\mathcal{Z}_{q}=\widetilde{\mathcal{Z}}_{q}$.
		\item[$\rm{(ii)}$] $q$ is called a \emph{decaying state} if 
$		|q(x)|\lesssim \langle x \rangle^{-4}$ for all $x\in \RR$.
	\end{enumerate}
\end{definition}
\begin{remark}
The existence of non-degenerate sign-changing states was proved in~\cite{PMPP1,PMPP2,MW}.
As noticed in \cite[Remark 3.2]{DKMASNSP},
the existence of \emph{non-degenerate, decaying states} then follows by translation and the Kelvin transformation.
\end{remark}

The main goal of this article is to construct $N$-excited states.
\begin{theorem}[Existence of multi-excited states]\label{main:thm}
Let $N\ge 2$ and let ${\boldsymbol{\rm{e}}}$ be a unit vector of $\R^5$. For $n\in \{1,\ldots,N\}$, let $\bell_{n}\in \R^{5}$ be such that
\begin{equation*}
\mbox{$\bell_{n}=\ell_{n}{\boldsymbol{\rm{e}}}$ where $-1<\ell_{n}<1$ and 
for all $n\ne n'$, $\ell_{n}\ne \ell_{n'}$.}
\end{equation*}
Let $q_{1},\ldots,q_{N}$ be non-degenerate decaying states of~\eqref{Yamabe}. Then there exist $T_{0}>0$ and a solution $\vec{u}=(u,\partial_{t}u)$ of~\eqref{wave} in the energy space $\dot{H}^{1}\times L^{2}$, defined on $[T_{0},+\infty)$ such that
\begin{equation}\label{thm:est}
\lim_{t\to +\infty}\bigg\|\vec{u}(t)-\sum_{n=1}^{N}\vec{Q}_{n}(t)
\bigg\|_{\dot{H}^{1}\times L^{2}}=0,
\end{equation}
where for $n=1,\ldots,N$,
\begin{equation*}
\vec{Q}_{n}(t,x)=\left(
\begin{array}{c} 
q_{n,\bell_{n}}\left(x-\bell_{n} t\right)\\
-\ell_{n}\partial_{x_{1}}q_{n,\bell_{n}}\left(x-\bell_{n} t\right)
\end{array}\right).
\end{equation*}
\end{theorem}
\begin{remark}
By invariance by rotation in $\R^5$, we assume without loss of generality that $\boldsymbol{\rm{e}}$ is $\boldsymbol{\rm{e}}_1$, the first vector of the canonical basis of $\R^5$.
\end{remark}
\begin{remark}
The \emph{non-degeneracy} condition ensures that the null space of $\mathcal{L}$ is generated by 21-parameter transformations, which is needed in our proof to obtain cancellation properties for the linearized operator around the excited states. See more details in \S\ref{SS:trans} and \S\ref{SS:Spectral}.
 \end{remark}
\begin{remark}
The \emph{decay} condition ensures that the  nonlinear interactions between two excited states of different speeds is of order at most $t^{-4}$. This rate allows us to close the energy estimates (see more details in \S\ref{SS:Boot} and \S\ref{sec:ener}). 
At this point, we do not known how to prove Theorem~\ref{main:thm} without this condition.
In particular, we do not construct multi-solitons partly based on \emph{non-degenerate decaying state} and ground state, since the interaction caused by the ground state would be $t^{-3}$.
\end{remark}
\begin{remark}
Using the Lorentz transformation, the existence result extends to the case of $2$-solitons for any different, possibly noncollinear speeds $\bell_{1}$ , $\bell_{2}$. See~\cite[Section 5]{MMwave1}.
\end{remark}

The constructions of \emph{asymptotic multi-solitons} for dispersive and wave equations have been the subject of several previous works, for both stable and unstable solitons. First results in non-integrable contexts were given by Merle~\cite{Mnls} for the $L^{2}$ critical nonlinear Schr\"odinger equation (NLS) and Martel~\cite{Ma} for the subcritical and critical generalized Korteweg-de Vries equations (gKdV). Next,
the strategy of these works was extended to the case of exponentially unstable solitons: see
C\^ote, Martel and Merle~\cite{CMM} for the construction of multi-solitons for supercritical (gKdV) and
(NLS), and Combet~\cite{Com} for a classification result for supercritical (gKdV). We refer to~\cite{CMakg,MMwave1} for results on the existence of multi-solitons for the nonlinear Klein-Gordon (NLKG) and 5D energy critical wave equation that inspired the present work. See also~\cite{CMkg,MMnls,MRT,XY} for other existence results.

The article is organized as follows. Section~\ref{S:2} describes the spectral theory for any \emph{non-degenerate decaying state} $q$. Section~\ref{S:3} introduces technical tools involved in a dynamical approach to the $N$-soliton problem for~\eqref{wave}: estimates of the non-linear interactions between solitons, decomposition by modulation and parameter estimates. Finally, Theorem~\ref{main:thm} is proved in Section~\ref{S:4} by energy estimates and a suitable compactness argument.
\subsection{Notation}
We denote
\begin{equation*}
(u,v)_{L^{2}}=\int_{\R^{5}}uv\d x,\quad (u,v)_{\dot{H}^{1}}=\int _{\R^{5}}(\nabla u\cdot\nabla v)\d x.
\end{equation*}
For
\begin{equation*}
\vec{u}=\left(\begin{array}{c}u_{1} \\u_{2}\end{array}\right),\quad \vec{v}=\left(\begin{array}{c}v_{1} \\v_{2}\end{array}\right),
\end{equation*}
denote
\begin{equation*}
\big(\vec{u},\vec{v})_{L^{2}}=\sum_{k=1,2}\big(u_{k},v_{k})_{L^{2}},\quad 
\|\vec{u}\|_{L^{2}}=\left(\vec{u},\vec{u}\right)_{L^{2}},
\end{equation*}
\begin{equation*}
\big(\vec{u},\vec{v})_{\E}=\big(u_{1},v_{1})_{\dot{H}^{1}}+(u_{2},v_{2})_{L^{2}},\quad 
\|\vec{u}\|_{\E}=\big(\vec{u},\vec{u})_{\E}.
\end{equation*}
When $x_{1}$ is seen as a specific coordinate, denote
\begin{equation*}
x'=(x_{2},x_{3},x_{4}, x_{5}),\quad
\overline{\nabla}=\left(\partial_{x_{2}},\partial_{x_{3}},\partial_{x_{4}},\partial_{x_{5}}\right),\quad\mbox{and}\quad
\bar{\Delta}=\sum_{i=2}^{5}\partial_{x_{i}}^{2}.
\end{equation*}
For $-1<\ell<1$, set 
\begin{equation*}
x_{\ell}=\left(\frac{x_{1}}{\sqrt{1-\ell^{2}}},x'\right)\quad \mbox{for}\ x=(x_{1},x')\in \R^{5}.
\end{equation*}
Let $X_{1}(\R^{5})=\dot{H}^{2}(\R^{5})\cap \dot{H}^{1}(\R^{5})$ and $X=X_{1}(\R^{5})\times H^{1}(\R^{5})$.

Denote by $\mathcal{O}_{5}$ be the orthogonal group in dimension 5. Let $\mathcal{SO}_{5}$ be the special orthogonal group, \emph{i.e.} the subgroup of the elements of $\mathcal{O}_{5}$ with determinant 1.
\subsection*{Acknowledgements}
The author wants to thank his advisor Yvan Martel for his
constant help and unfailing encouragement throughout the completion of this work.

\section{Spectral theory for non-degenerate state}\label{S:2}
\subsection{Transforms of stationary solutions}\label{SS:trans}
Following~\cite{DKMASNSP}, we recall that~\eqref{Yamabe} is invariant under the following four transformations:

\begin{enumerate}
	\item[$\rm{(1)}$] translations: If $q\in \Sigma$ then $q(x+a)\in \Sigma$, for all $a\in \RR$;
	\item[$\rm{(2)}$] dilation: If $q\in \Sigma$ then $\lambda^{\frac{3}{2}}q(\lambda x)\in \Sigma$ for all $\lambda>0$;
	\item[$\rm{(3)}$] orthogonal transformation: If $q\in \Sigma$ then $q(Px)\in \Sigma$ where $P\in \mathcal{O}_{5}$;
	
\item[$\rm{(4)}$] Kelvin transformation: If $q\in \Sigma$ then $|x|^{-3}q\big(\frac{x}{|x|^{2}}\big)\in \Sigma$.
\end{enumerate}
Let $M$ be the group of one-to-one maps of $\RR$ generated by the above four transformations.
Then, from \cite[Section 3]{DKMASNSP}, $M$ generates a $21$-parameter family of transformations in a neighborhood of the identity. More precisely, we give explictly the formula for this $21$-parameter family of transformations. Set the one-to-one map 
\begin{equation*}
\tau: \left\{ (i,j)\in \mathbb{N}^{2}: 1\le i<j\le 5 \right\}\to \left\{1,2,\ldots,10\right\},
\end{equation*} 
with
\begin{equation*}
\tau(i,j)=3i+j-\frac{(i-1)(i-2)}{2}-4.
\end{equation*}
For $c=\left(c_{1},c_{2},\ldots,c_{10}\right)$, set
\begin{equation*}
A_{c}=\left[a^{c}_{i,j}\right]_{1\le i\le j\le 5},\quad e^{A_{c}}=\sum_{n=0}^{\infty}\frac{A^{n}_{c}}{n!}\in \mathcal{SO}_{5},
\end{equation*}
where $a_{i,i}^{c}=0$, $a_{i,j}=c_{\tau(i,j)}$ if $i<j$ and $a_{i,j}=-c_{\tau(j,i)}$ if $j<i$. Note that
this definition provides a parametrization of $\mathcal{SO}_{5}$ by $\R^{10}$ 
in a neighborhood of the identity matrix.

For $\mathcal{A}=\left(\lambda,\xi,a,c\right)\in (0,\infty)\times \RR\times\RR\times \R^{10}$, we introduce the following transform $\theta_{\mathcal{A}}\in M$,
\begin{equation*}
\theta_{\mathcal{A}}(f)(x)=\lambda^{\frac{3}{2}}\left|\frac{x}{|x|}-a|x|\right|^{-3}
f\left(\xi+\frac{\lambda e^{A_{c}}\left(x-a|x|^{2}\right)}{1-2\langle a,x\rangle +|a|^{2}|x|^{2}}\right)\quad 
\mbox{for all}\ f\in \dot{H}^{1}.
\end{equation*}
Observe that for all $q\in \Sigma$, $\widetilde{\mathcal{Z}}_{q}$ is generated by taking partial derivatives of $\theta_{\mathcal{A}}(q)$ with respect to $\mathcal{A}=(\lambda,\xi,a,c)$ at $\mathcal{A}=(1,0,0,0)$ (see details in~\cite[Lemma 3.8]{DKMASNSP}).

\subsection{Spectral analysis of  $\mathcal{L}$}\label{SS:Spectral}
Following~\cite{DKMASNSP}, we gather some spectral properties of the linearized operator.
For $q$ be a non-degenerate decaying state and collinear speeds $\bell=\ell {\be}_{1}$, let 
\begin{equation}\label{equ:Ql}
q_{\ell}=q(x_{\ell})=q\left(\frac{x_{1}}{\sqrt{1-\ell^{2}}},x'\right),\quad 
-(1-\ell^{2})\partial^{2}_{x_{1}}q_{\ell}-\bar{\Delta} q_{\ell}+|q_{\ell}|^{\frac{4}{3}}q_{\ell}=0.
\end{equation}
Define the following operator,
\begin{equation}\label{defL}
\mathcal{L}=-\Delta-f'(q),\quad  \mathcal{L}_{\ell}=-(1-\ell^{2})\partial^{2}_{x_{1}}-\bar{\Delta}-f'(q_{\ell}),
\end{equation}
and
\begin{equation}\label{defH}
\mathcal{H}_{\ell}=\left(\begin{array}{cc}-\Delta-f'(q_{\ell}) & -\ell\partial_{x_{1}} \\\ell\partial_{x_{1}} & 1\end{array}\right),\quad 
{\mathrm{J}}=\left(\begin{array}{cc} 0 & 1 \\ -1 & 0\end{array}\right).
\end{equation}

\begin{lemma}[Spectral properties of $\mathcal{L}$]\label{le:Q}
	\emph{(i) Spectrum}.
	The self-adjoint operator $\mathcal{L}$ has essential spectrum $[0,+\infty)$, a finite number $J\ge 1$ of negative eigenvalues $($counted with multiplicity$)$, and its kernel is 
	$\widetilde{\mathcal{Z}}_{q}$. Let $(Y_{j})_{j\in\{1,\ldots,J\}}$ be an $L^{2}$ orthogonal family of eigenvectors of $\mathcal{L}$ corresponding to the eigenvalues $(-\lambda_{j}^{2})_{j\in \{1,\ldots,J\}}$, $i.e.$ for $j,j'\in \{1,\ldots,J\}$
	\begin{equation*}
	\mathcal{L}Y_{j}=-\lambda_{j}^{2}Y_{j},\quad \lambda_{j}>0\quad 
	\mbox{and}\quad \left(Y_{j},Y_{j'}\right)_{L^{2}}=\delta_{jj'}.
	\end{equation*}
	Let $\Psi_{k}$ be an $\dot{H}^{1}$-orthogonal basis of ${\rm{Ker}}\mathcal{L}$, $i.e.$ for any $k,k'\in \{1,\ldots,K\}$
	\begin{equation*}
	(\Psi_{k},\Psi_{k'})_{\dot{H}^{1}}=\delta_{kk'}\quad  \mbox{and}\quad {\rm{Span}}(\Psi_{1},\ldots,\Psi_{K})=\widetilde{\mathcal{Z}}_{q}.
	\end{equation*}
	It holds, for all $k=1,\ldots,K$, $j=1,\ldots,J$ and $\alpha\in \mathbb{N}^{5}$ with $|\alpha|\le 2$, on $\RR$,
	\begin{equation*} 
	\left|\partial_{x}^{\alpha}Y_{j}(x)\right|\lesssim e^{-\lambda_{j}|x|}\quad \mbox{and}\quad 
	\left|\partial_{x}^{\alpha}\Psi_{k}(x)\right|\lesssim \langle x\rangle^{-(3+\alpha)}.
	\end{equation*}
	
	\emph{(ii) Nonnegativity under $(Y_{j})_{j=1,\ldots,J}$ orthogonality.}
It holds	\begin{equation*}
	(\mathcal{L}v,v)_{L^{2}}\ge 0\quad \mbox{for all $ v\in N^{\perp}$,}
	\end{equation*}
	where
	\begin{equation*}
	N^{\perp}=\left\{v\in \dot{H}^{1}:(v,Y_{j})_{L^{2}}=0,\ \mbox{for any}\ j=1,\ldots,J\right\}.
	\end{equation*}
	\emph{(iii) Cancellation.} It holds
	\begin{equation*}
	\int_{\RR}f''(q) \Psi_{k}  \Psi_{k'}\Psi_{k''}\d x=0\quad \mbox{for all}\ k,k',k''= 1,\ldots,K.
	\end{equation*}
\end{lemma}
\begin{proof} Proof of (i).
	For the spectral properties of $\mathcal{L}$, see the proof of \cite[Claim 3.5]{DKMASNSP}.
	The algebraic decay of kernel functions $\Psi_{k}$ directly follows from the non-degenerate decaying condition.
	The exponential decay of negative functions $Y_{j}$ follows from standard elliptic arguments. See \emph{e.g.}~\cite{Mes} and~\cite[Proposition 3.9]{DKMASNSP}.
	
	Proof of (ii). The proof relies on a standard argument based on the min-max principle (see \emph{e.g.}~\cite[Proposition 3.6]{DKMASNSP}) and we omit it.
	
	Proof of (iii). Without loss of generality, we consider 
	\begin{equation*}
	\quad \Psi_{k''}=-3x_{i}q+|x|^{2}\partial_{x_{i}}q-2x_{i}x\cdot\nabla q.
	\end{equation*}
	For all $k'=1,\ldots,K$, we have
	\begin{equation*}
	-\Delta\Psi_{k'}-f'(q)\Psi_{k'}=0.
	\end{equation*}
	We consider the transformation $T_{a}=\theta_{\mathcal{A}}$ with $\theta_{\mathcal{A}}=(1,0,a,0)$ for the above identity,
    \begin{equation*}
    -\Delta T_{a}\Psi_{k'}-f'(T_{a}q)T_{a}\Psi_{k'}=0.
    \end{equation*}
    Taking the derivative of above identity with respect to $a_{i}$, and then letting $a=0$,
    we obtain
    \begin{equation*}
    f''(q)\Psi_{k'}\Psi_{k''}=-\Delta \tilde{\Psi}_{k'}-f'(q)\tilde{\Psi}_{k'}=\mathcal{L} \tilde{\Psi}_{k'},
    \end{equation*}
    where
    \begin{equation*}
    \tilde{\Psi}_{k'}=-3x_{i}\Psi_{k'}+|x|^{2}\partial_{x_{i}}\Psi_{k'}-2x_{i}(x\cdot\nabla \Psi_{k'}).
    \end{equation*}
    Therefore, by integration by parts, for any $k=1,\ldots,K$,
    \begin{equation*}
     \int_{\RR}\Psi_{k}\left(f''(q)\Psi_{k'}\Psi_{k''}\right)\d x
     =\int_{\RR}\Psi_{k} \left(\mathcal{L}\tilde{\Psi}_{k'}\right)\d x 
     =\int_{\RR}\left(\mathcal{L}\Psi_{k}\right) \tilde{\Psi}_{k'}\d x=0.
    \end{equation*}
    Using the non-degenerate condition and proceeding similarly for all the parameters in the transformation $\theta_{\mathcal{A}}$, we complete the proof of (iii).
    \end{proof}
\subsection{Coercivity of $\mathcal{H}_{\ell}$}\label{SS:Coer}
For $-1<\ell<1$, let
\begin{equation*}
\vec{\Psi}_{k,\ell}=\left(
\begin{array}{c}\Psi_{k}(x_{\ell})\\ -\ell\partial_{x_{1}}\Psi_{k}(x_{\ell}),
\end{array}\right)\quad \mbox{for}\ k=1,\ldots,K,
\end{equation*}
and for $j=1,\ldots,J$
\begin{equation*}
Y_{j,\ell}=Y_{j}(x_{\ell}),\quad \vec{Y}_{j,\ell}^{0}=\left(\begin{array}{c} Y_{j,\ell} \\ 0 \end{array}\right),\quad 
\vec{Y}_{j,\ell}^{\pm}=\left(\begin{array}{c}
Y_{j,\ell}e^{\mp\frac{\ell\lambda_{j}}{\sqrt{1-\ell^{2}}}x_{1}}\\
-(\ell\partial_{x_{1}}Y_{j,\ell}\mp \frac{\lambda_{j}}{\sqrt{1-\ell^{2}}}Y_{j,\ell})e^{\mp\frac{\ell\lambda_{j}}{\sqrt{1-\ell^{2}}}x_{1}}
\end{array}
\right).
\end{equation*}
Following \cite{CM_erratum}, we define the function $\vec{W}_{j,\ell}$ by
\begin{equation*}
\vec{W}_{j,\ell}=\vec{Y}^{+}_{j,\ell}+\vec{Y}_{j,\ell}^{-},\quad \mbox{for}\ j=1,\ldots,J.
\end{equation*}

First, we prove the following technical identities of $\vec{Y}^{\pm}_{j,\ell}$ and $\vec{W}_{j,\ell}$.
\begin{lemma} 
	The functions $\vec{Y}^{\pm}_{j,\ell}$ and $\vec{W}_{j,\ell}$ satisfy the following properties.
    \begin{enumerate}
    	\item For $j=1,\ldots,J$, 	\begin{equation}\label{HY}
    	\mathcal{H}_{\ell}\vec{Y}_{j,\ell}^{\pm}=
    	\mp {\lambda_{j}}(1-\ell^{2})^{\frac{1}{2}}{\rm{J}}\vec{Y}_{j,\ell}^{\pm}.
    	\end{equation}
	\item For $j=1,\ldots,J$, 
	\begin{equation}\label{HWW}
	\left(\mathcal{H}_{\ell}\vec{W}_{j,\ell},\vec{W}_{j,\ell}\right)_{L^{2}}=-{4\lambda_{j}^{2}}(1-\ell^{2})^{\frac{1}{2}}.
	\end{equation}
	\item For $j,j'=1,\ldots,J$ with $j\ne j'$,
	\begin{equation}\label{HWw}
	\left(\mathcal{H}_{\ell}\vec{W}_{j,\ell},\vec{W}_{j',\ell}\right)_{L^{2}}=0.
	\end{equation}
\end{enumerate}
\end{lemma}
\begin{proof} 
	Proof of (i). On the one hand, from (i) of Lemma~\ref{le:Q} and direct computation,
	\begin{equation*}
	\begin{aligned}
	&\left(-\Delta-f'(q_{\ell})\right)Y_{j,\ell}e^{\mp\frac{\ell\lambda_{j}}{\sqrt{1-\ell^{2}}}x_{1}}
	+\ell\partial_{x_{1}}\left(\left(\ell\partial_{x_{1}}Y_{j,\ell}\mp \frac{\lambda_{j}}{\sqrt{1-\ell^{2}}}Y_{j,\ell}\right)
	e^{\mp\frac{\ell\lambda_{j}}{\sqrt{1-\ell^{2}}}x_{1}}\right)\\
	&=\left(\mathcal{L}_{\ell}Y_{j,\ell}\right)e^{\mp\frac{\ell\lambda_{j}}{\sqrt{1-\ell^{2}}}x_{1}}
	\pm \ell(1-\ell^{2})^{\frac{1}{2}}\lambda_{j}(\partial_{x_{1}}Y_{j,\ell})
	e^{\mp\frac{\ell\lambda_{j}}{\sqrt{1-\ell^{2}}}x_{1}}\\
	&=\pm \lambda_{j}(1-\ell^{2})^{\frac{1}{2}}\left(\ell\partial_{x_{1}}Y_{j,\ell}\mp\frac{\lambda_{j}}
	{\sqrt{1-\ell^{2}}}Y_{j,\ell}\right)e^{\mp\frac{\ell\lambda_{j}}{\sqrt{1-\ell^{2}}}x_{1}}.
	\end{aligned}
	\end{equation*}
	On the other hand, by direct computation,
	\begin{equation*}
	\begin{aligned}
	&\ell\partial_{x_{1}}\left(Y_{j,\ell}e^{\mp\frac{\ell\lambda_{j}}{\sqrt{1-\ell^{2}}}x_{1}}\right)
	-\left(\ell \partial_{x_{1}}Y_{j,\ell}\mp \frac{\lambda_{j}}{\sqrt{1-\ell^{2}}}Y_{j,\ell}\right)e^{\mp\frac{\ell\lambda_{j}}{\sqrt{1-\ell^{2}}}x_{1}}\\
	&=\left(\pm \frac{\lambda_{j}}{\sqrt{1-\ell^{2}}}\mp\frac{\ell^{2}\lambda_{j}}{\sqrt{1-\ell^{2}}}\right)Y_{j,\ell}
	e^{\mp\frac{\ell\lambda_{j}}{\sqrt{1-\ell^{2}}}x_{1}}=\pm\lambda_{j}(1-\ell^{2})^{\frac{1}{2}}Y_{j,\ell}
	e^{\mp\frac{\ell\lambda_{j}}{\sqrt{1-\ell^{2}}}x_{1}}.
	\end{aligned}
	\end{equation*}
	Gathering above identities, we obtain~\eqref{HY}.
	
	Proof of (ii). To prove~\eqref{HWW}, we first show that for $j,j'=1,\ldots,J$,
	\begin{equation}\label{HYY}
	\left(\mathcal{H}_{\ell}\vec{Y}_{j,\ell}^{+},\vec{Y}_{j',\ell}^{+}\right)_{L^{2}}=
	\left(\mathcal{H}_{\ell}\vec{Y}_{j,\ell}^{-},\vec{Y}_{j',\ell}^{-}\right)_{L^{2}}=0.
	\end{equation}
	On the one hand, by~\eqref{HY}, for $j,j'=1,\ldots,J$,
	\begin{equation*}
	\left(\mathcal{H}_{\ell}\vec{Y}_{j,\ell}^{+},\vec{Y}_{j',\ell}^{+}\right)_{L^{2}}
	=-\lambda_{j}(1-\ell^{2})^{\frac{1}{2}}\left({\rm{J}}\vec{Y}_{j,\ell}^{+},\vec{Y}_{j',\ell}^{+}\right)_{L^{2}}.
	\end{equation*}
	On the other hand, using again~\eqref{HY}, for $j,j'=1,\ldots,J$,
	\begin{equation*}
	\begin{aligned}
	\left(\mathcal{H}_{\ell}\vec{Y}_{j,\ell}^{+},\vec{Y}_{j',\ell}^{+}\right)_{L^{2}}
	=\left(\vec{Y}_{j,\ell}^{+},\mathcal{H}_{\ell}\vec{Y}_{j',\ell}^{+}\right)_{L^{2}}
=\lambda_{j'}(1-\ell^{2})^{\frac{1}{2}}\left({\rm{J}}\vec{Y}_{j,\ell}^{+},\vec{Y}_{j',\ell}^{+}\right)_{L^{2}}.
	\end{aligned}
	\end{equation*}
	Since $\lambda_{j},\lambda_{j'}>0$, this implies $	\left(\mathcal{H}_{\ell}\vec{Y}_{j,\ell}^{+},\vec{Y}_{j',\ell}^{+}\right)_{L^{2}}=0$. The proof of~\eqref{HYY} for $\left(\mathcal{H}_{\ell}\vec{Y}_{j,\ell}^{-},\vec{Y}_{j',\ell}^{-}\right)_{L^{2}}=0$ follows from similar arguments and it is omitted.
	
	Now, we start to prove~\eqref{HWW}. From~\eqref{HYY}, we have
	\begin{equation*}
	\left(\mathcal{H}_{\ell}\vec{W}_{j,\ell},\vec{W}_{j,\ell}\right)_{L^{2}}=
	2\left(\mathcal{H}_{\ell}\vec{Y}^{+}_{j,\ell},\vec{Y}^{-}_{j,\ell}\right)_{L^{2}}.
	\end{equation*}
	Using~\eqref{HY} and the explicit expression of $\vec{Y}_{j,\ell}^{\pm}$, we compute
	\begin{equation*}
	\left(\mathcal{H}_{\ell}\vec{Y}^{+}_{j,\ell},\vec{Y}^{-}_{j,\ell}\right)_{L^{2}}
	=-\lambda_{j}(1-\ell^{2})^{\frac{1}{2}}\left({\rm{J}}\vec{Y}_{j,\ell}^{+},\vec{Y}_{j,\ell}^{-}\right)_{L^{2}}
	=-2\lambda_{j}^{2}(1-\ell^{2})^{\frac{1}{2}}\left(Y_{j},Y_{j}\right)_{L^{2}}.
	\end{equation*}
	Therefore, the identity~\eqref{HWW} follows from the normalization $(Y_{j},Y_{j})_{L^{2}}=1$.
	
	Proof of (iii). To prove~\eqref{HWw}, we first show that for $j,j'=1,\ldots,J$ with $j\ne j'$,
	\begin{equation}\label{HYy1}
	\left(\mathcal{H}_{\ell}\vec{Y}_{j,\ell}^{+},\vec{Y}_{j',\ell}^{-}\right)_{L^{2}}=0\quad \mbox{when}\ \ \lambda_{j}\ne \lambda_{j'},
	\end{equation}
	and
	\begin{equation}\label{HYy2}
	\left(\mathcal{H}_{\ell}\vec{Y}_{j,\ell}^{+},\vec{Y}_{j',\ell}^{-}\right)_{L^{2}}+
	\left(\mathcal{H}_{\ell}\vec{Y}_{j',\ell}^{+},\vec{Y}_{j,\ell}^{-}\right)_{L^{2}}
	=0\quad \mbox{when}\ \ \lambda_{j}=\lambda_{j'}.
	\end{equation}
	Let $j\ne j'$. In the case where $\lambda_{j}\ne \lambda_{j'}$. On the one hand, by~\eqref{HY}, for
	$j\ne j'$,
	\begin{equation*}
	\left(\mathcal{H}_{\ell}\vec{Y}_{j,\ell}^{+},\vec{Y}_{j',\ell}^{-}\right)_{L^{2}}
	=-\lambda_{j}(1-\ell^{2})^{\frac{1}{2}}\left({\rm{J}}\vec{Y}_{j,\ell}^{+},\vec{Y}_{j',\ell}^{-}\right)_{L^{2}}.
	\end{equation*}
	On the other hand, using again~\eqref{HY},
	\begin{equation*}
	\left(\mathcal{H}_{\ell}\vec{Y}_{j,\ell}^{+},\vec{Y}_{j',\ell}^{-}\right)_{L^{2}}
	=	\left(\vec{Y}_{j,\ell}^{+},\mathcal{H}_{\ell}\vec{Y}_{j',\ell}^{-}\right)_{L^{2}}
	=-\lambda_{j'}(1-\ell^{2})^{\frac{1}{2}}\left({\rm{J}}\vec{Y}_{j,\ell}^{+},\vec{Y}_{j',\ell}^{-}\right)_{L^{2}}.
	\end{equation*}
	Since $\lambda_{j}\ne \lambda_{j'}$, this implies~\eqref{HYy1}.
	
	In the case where $\lambda_{j}=\lambda_{j'}$. From~\eqref{HY} and the 
	explicit expression of $\vec{Y}_{j,\ell}^{\pm}$, $\left(Y_{j},Y_{j'}\right)_{L^{2}}=0$ and integration by parts, we compute
	\begin{equation*}
	\left(\mathcal{H}_{\ell}\vec{Y}_{j,\ell}^{+},\vec{Y}_{j',\ell}^{-}\right)_{L^{2}}
	=-\lambda_{j}(1-\ell^{2})^{\frac{1}{2}}\left({\rm{J}}\vec{Y}_{j,\ell}^{+},\vec{Y}_{j',\ell}^{-}\right)_{L^{2}}
	=-2\ell\lambda_{j}(1-\ell^{2})^{\frac{1}{2}}\left(\partial_{x_{1}}Y_{j},Y_{j'}\right)_{L^{2}}.
	\end{equation*}
	Therefore, by~\eqref{HYY} and integration by parts,
	\begin{equation*}
	\left(\mathcal{H}_{\ell}\vec{Y}_{j,\ell}^{+},\vec{Y}_{j',\ell}^{-}\right)_{L^{2}}
	+\left(\mathcal{H}_{\ell}\vec{Y}_{j',\ell}^{+},\vec{Y}_{j,\ell}^{-}\right)_{L^{2}}
	=0,
	\end{equation*}
	which proves~\eqref{HYy2}.
	
	Now, we start to prove~\eqref{HWw}. Let $j\ne j'$. From~\eqref{HYY},
	\begin{equation*}
	\left(\mathcal{H}_{\ell}\vec{W}_{j,\ell},\vec{W}_{j',\ell}\right)_{L^{2}}
	=\left(\mathcal{H}_{\ell}\vec{Y}_{j,\ell}^{+},\vec{Y}_{j',\ell}^{-}\right)_{L^{2}}
	+\left(\mathcal{H}_{\ell}\vec{Y}_{j',\ell}^{+},\vec{Y}_{j,\ell}^{-}\right)_{L^{2}}.
	\end{equation*}
	Thus, the identity~\eqref{HWw} directly follows from~\eqref{HYy1} and~\eqref{HYy2}.
\end{proof}

Second, we prove the following coercivity property.
\begin{proposition}\label{procoer}
	There exist $\mu>0$ such that, for all $\vec{v}\in \dot{H}^{1}\times L^{2}$
	\begin{equation}\label{coer}
	\left(\mathcal{H}_{\ell}\vec{v},\vec{v}\right)_{L^{2}}
	\ge \mu \|\vec{v}\|^{2}_{\E}-
	\mu^{-1}\left(\sum_{k=1}^{K}\left(\vec{v},\vec{\Psi}_{k,\ell}\right)^{2}_{\E}
	+\sum_{\pm,j=1}^{J}\left(\mathcal{H}_{\ell}\vec{v},\vec{Y}_{j,\ell}^{\pm}\right)_{L^{2}}^{2}\right).
	\end{equation}
\end{proposition}

\begin{proof}
	By a standard argument, it suffices to prove that there exists $\mu>0$ such that 
	for any $\vec{v}\in \dot{H}^{1}\times L^{2}$ with orthogonality conditions $\left(\mathcal{H}_{\ell}\vec{v},\vec{W}_{j,\ell}\right)_{L^{2}}=\left(\vec{v},\vec{\Psi}_{k,\ell}\right)_{\E}=0$ for all $j=1,\ldots,J$ and $k=1,\ldots,K$ there holds
	\begin{equation}\label{coer1}
	\left(\mathcal{H}_{\ell}\vec{v},\vec{v}\right)_{L^{2}}\ge \mu \|\vec{v}\|_{\E}^2.
	\end{equation}
	
	\textbf{Step 1.} Negative direction. Note that, for $\vec{v}=(v,z)\in \dot{H}^{1}\times L^{2}$,
	\begin{equation}\label{eq:Hv}
	\mathcal{H}_{\ell}\vec{v}=\left(\begin{array}{c}
	-\Delta v-f'(q_{\ell})v-\ell\partial_{x_{1}}z \\ \ell \partial_{x_{1}}v+z\end{array}\right),\ \ 
	\left(\mathcal{H}_{\ell}\vec{v},\vec{v}\right)_{L^{2}}=
	\left(\mathcal{L}_{\ell}v,v\right)_{L^{2}}+\|\ell\partial_{x_{1}}v+z\|_{L^{2}}^{2}.
	\end{equation}
	Therefore, from (ii) of Lemma~\ref{le:Q}, for any $\vec{v}\in \dot{H}^{1}\times L^{2}$,
	\begin{equation}\label{Hge0}
	\left(\vec{v},\vec{Y}_{j,\ell}^{0}\right)_{L^{2}}=0\ \mbox{for}\ j=1,\ldots,J
	\implies \left(\mathcal{H}_{\ell}\vec{v},\vec{v}\right)_{L^{2}}\ge 0.
	\end{equation}
	We claim, for any $\vec{v}\in \dot{H}^{1}\times L^{2}$,
	\begin{equation}\label{Hge01}
	\left(\mathcal{H}_{\ell}\vec{v},\vec{W}_{j,\ell}\right)_{L^{2}}=0\ \mbox{for}\ j=1,\ldots,J
	\implies \left(\mathcal{H}_{\ell}\vec{v},\vec{v}\right)_{L^{2}}\ge 0.
	\end{equation}
	For the sake of contradiction, assume that there exists a vector $\vec{v}\in \dot{H}^{1}\times L^{2}$ satisfying the orthogonality conditions in~\eqref{Hge01} and $\left(\mathcal{H}_{\ell}\vec{v},\vec{v}\right)<0$. For any $(c_{j})_{j=0,\ldots,J}\in \mathbb{R}^{J+1}/{\bf{0}}$, set 
	\begin{equation*}
	\vec{h}=c_{0}\vec{v}+\sum_{j=1}^{J}c_{j}\vec{W}_{j,\ell}\in {\rm{Span}}\left(\vec{v},\vec{W}_{1,\ell},\ldots,\vec{W}_{J,\ell}\right).
	\end{equation*}
	By direct computation,~\eqref{HWW},~\eqref{HWw} and the orthogonality conditions in~\eqref{Hge01},
	\begin{equation*}
	\begin{aligned}
	\left(\mathcal{H}_{\ell}\vec{h},\vec{h}\right)_{L^{2}}
	&=c_{0}^{2}\left(\mathcal{H}_{\ell}\vec{v},\vec{v}\right)_{L^{2}}+\sum_{j=1}^{J}c_{j}^{2}
	\left(\mathcal{H}_{\ell}\vec{W}_{j,\ell},\vec{W}_{j,\ell}\right)_{L^{2}}\\
	&=c_{0}^{2}\left(\mathcal{H}_{\ell}\vec{v},\vec{v}\right)_{L^{2}}-4(1-\ell^{2})^{\frac{1}{2}}
	\sum_{j=1}^{J}c_{j}^{2}\lambda_{j}^{2}<0.
	\end{aligned}
	\end{equation*}
	It follows that $\left(\mathcal H_{\ell}\cdot,\cdot\right)_{L^{2}}<0$ on ${\rm{Span}}\left(\vec{v},\vec{W}_{1,\ell},\ldots,\vec{W}_{J,\ell}\right)$ and
	\begin{equation*}
	{\rm{dim}}\SPAN\left(\vec{v},\vec{W}_{1,\ell},\ldots,\vec{W}_{J,\ell}\right)=J+1.
	\end{equation*}
	This is contradictory with~\eqref{Hge0} which says that 
	$\left(\mathcal{H}_{\ell}\cdot,\cdot\right)_{L^{2}}$ is nonnegative under only $J$ orthogonality conditions.
	The proof of~\eqref{Hge01} is complete.
	
	{\bf{Step 2}.} Null direction. Note that, from~\eqref{eq:Hv}, for any $\vec{v}=(v,z)\in {\rm{Ker}} \mathcal{H}_{\ell}$,
	\begin{equation*}
	-\Delta v-f'(q_{\ell})v-\ell \partial_{x_{1}}z=0,\quad \ell \partial_{x_{1}}v+z=0,
	\end{equation*}
	which is equivalent to
	\begin{equation*}
	-\Delta v+\ell^{2}\partial_{x_{1}}^{2}v-f'(q_{\ell})v=0,\quad z=-\ell \partial_{x_{1}}v.
	\end{equation*}
Thus,
	\begin{equation*}
	{\rm{Ker}} \mathcal{H}_{\ell}=\SPAN\left(\vec{\Psi}_{k,\ell},k=1,\ldots,K\right).
	\end{equation*}
	We claim, for all $\vec{v}=(v,z)\in \dot{H}^{1}\times L^{2}$ with orthogonality conditions
	$(\mathcal{H}_{\ell}\vec{v},\vec{W}_{j,\ell})_{L^{2}}=(\vec{v},\vec{\Psi}_{k,\ell})_{\E}=0$ for $j=1,\ldots,J$ and $k=1,\ldots,K$,
	\begin{equation}\label{H>0}
	\vec{v}=0\quad \mbox{or}\quad (\mathcal{H}_{\ell}\vec{v},\vec{v})>0.
	\end{equation}
	Denote 
	\begin{equation*}
	\mathcal{N}_{\ell}^{\perp}=\left\{\vec{v}=(v,z)\in \dot{H}^{1}\times L^{2}: \left(\mathcal{H}_{\ell}\vec{v},\vec{W}_{j,\ell}\right)_{L^{2}}=0 \ \ \mbox{for}\ j=1,\ldots,J\right\}.
	\end{equation*}
	Indeed, it suffices to prove that for any $\vec{v}\in \mathcal{N}_{\ell}^{\perp}$
	\begin{equation*}
	(\mathcal{H}_{\ell}{\vec{v}},\vec{v})_{L^{2}}=0\implies \vec{v}\in {\rm{Ker}}\mathcal{H}_{\ell}.
	\end{equation*}
	Fix $\vec{v}\in \mathcal{N}_{\ell}^{\perp}$ with $(\mathcal{H}_{\ell}\vec{v},\vec{v})_{L^{2}}=0$.
	From $(\mathcal{H}_{\ell}\cdot,\cdot)_{L^{2}}\ge 0$ on $\mathcal{N}_{\ell}^{\perp}$ and the Cauchy-Schwarz inequality for $(\mathcal{H}_{\ell}\cdot,\cdot)$,
	\begin{equation}\label{product=0}
	(\mathcal{H}_{\ell}\vec{v},\vec{f})=0\quad \mbox{for all}\ \vec{f}=(f,g)\in \mathcal{N}^{\perp}_{\ell}.
	\end{equation}
	For any $\vec{\tilde{v}}=(\tilde{v},\tilde{z})\in C_{0}^{\infty}(\RR)\times C_{0}^{\infty}(\RR)$, we decompose
	\begin{equation*}
	\vec{\tilde{v}}=\vec{f}+\sum_{j=1}^{J}\alpha_{j}\vec{W}_{j,\ell}\quad \mbox{where}\ \vec{f}\in \mathcal{N}_{\ell}^{\perp}\ \ \mbox{and}\ \ \alpha_{j}=-\frac{1}{4}\lambda_{j}^{-2}(1-\ell^{2})^{-\frac{1}{2}}\left(\mathcal{H}_{\ell}\vec{\tilde{v}},\vec{W}_{j,\ell}\right)_{L^{2}}.
	\end{equation*}
	Thus, from~\eqref{product=0},
	\begin{equation*}
	\left(\mathcal{H}_{\ell}\vec{v},\vec{\tilde{v}}\right)_{L^{2}}
	=\left(\mathcal{H}_{\ell}\vec{v},\vec{f}\right)_{L^{2}}+\sum_{j=1}^{J}\alpha_{j}\left(\mathcal{H}_{\ell}\vec{v},\vec{W}_{j,\ell}\right)_{L^{2}}=0.
	\end{equation*}
	It follows that $\mathcal{H}_{\ell}\vec{v}=0$ in the sense of distribution,~\emph{i.e.}~$\vec{v}\in {\rm{Ker}}\mathcal{H}_{\ell}$. The proof of~\eqref{H>0} is complete.
	
	{\textbf{Step 3.}} Conclusion. Now, we prove~\eqref{coer1} by contradiction and using standard compactness argument.
	For the sake of contradiction, assume that there exists a sequence 
	\begin{equation*}
	\left\{\vec{v}_{n}=(v_{n},z_{n})\right\}_{n\in \mathbb{N}}\in \dot{H}^{1}\times L^{2}
	\end{equation*}
	such that 
	\begin{equation}\label{orthon}
	\vec{v}_{n}\in \mathcal{N}_{\ell}^{\perp},\quad (\vec{v}_{n},\vec{\Psi}_{k,\ell})_{\E}=0\quad \mbox{for}\ k=1,\ldots,K,
	\end{equation}
	and
	\begin{equation}\label{est:Hvn}
	0<\left(\mathcal{H}_{\ell}\vec{v}_{n},\vec{v}_{n}\right)_{L^{2}}<\frac{1}{n}\|\vec{v}_{n}\|^{2}_{\E},\quad  
	\int_{\RR}f'(q_{\ell})v_{n}^{2}=1\quad \mbox{for any}\ n\in \mathbb{N}.
	\end{equation}
	From the above inequalities and~\eqref{eq:Hv},  the sequence $\left\{\vec{v}_{n}\right\}_{n\in \mathbb{N}}$ is bounded in $\dot{H}^{1}\times L^{2}$. Upon extracting a subsequence, we can assume 
	\begin{equation}\label{weakcon}
	\vec{v}_{n}=(v_{n},z_{n})\rightharpoonup \vec{v}=(v,z)\in \dot{H}^{1}\times L^{2}\quad \mbox{as}\ n\to \infty.
	\end{equation} 
	On the one hand, by the Rellich theorem, we have
	\begin{equation*}
	\int_{\RR}f'(q_{\ell})v^{2}\d x=\lim_{n\to \infty}\int_{\RR}f'(q_{\ell})v_{n}^{2}\d x=1,
	\end{equation*}
	which implies $\vec{v}\ne 0$.
	On the other hand, from $\vec{v}_{n}=(v_{n},z_{n})\rightharpoonup \vec{v}=(v,z)\in \dot{H}^{1}\times L^{2}$,
	\begin{equation*}
\begin{aligned}
   \int_{\RR}(\ell\partial_{x_{1}} v+z)^{2}\d x&\le \mathop{\underline{\lim}}_{n \to \infty} \int_{\RR}(\ell\partial_{x_{1}} v_{n}+z_{n})^{2}\d x,\\
	\int_{\RR}\left((1-\ell^{2})(\partial_{x_{1}}v)^{2}+|\overline \nabla v|^{2}\right)\d x&\le \mathop{\underline{\lim}}_{n \to \infty} \int_{\R^{5}}\left((1-\ell^{2})(\partial_{x_{1}}v_{n})^{2}+|\overline \nabla v_{n}|^{2}\right)\d x.
	\end{aligned}
	\end{equation*}
	Combining the above inequalities with~\eqref{eq:Hv} and~\eqref{est:Hvn},
	\begin{equation*}
	\begin{aligned}
	\left(\mathcal{H}_{\ell}\vec{v},\vec{v}\right)_{L^{2}}\le \mathop{\underline{\lim}}_{n \to \infty} \left(\mathcal{H}_{\ell}\vec{v}_{n},\vec{v}_{n}\right)_{L^{2}}\le \mathop{\underline{\lim}}_{n \to \infty} \frac{1}{n}\|\vec{v}_{n}\|^{2}_{\E}\le 0.
	\end{aligned}
	\end{equation*}
	Using again~\eqref{weakcon} and taking limit in~\eqref{orthon}, we have
	\begin{equation*}
	\vec{v}\in \mathcal{N}^{\perp},\quad  (\vec{v},\vec{\Psi}_{k,\ell})_{\E}=0\quad \mbox{for}\ k=1,\ldots,K,
	\end{equation*}
	Therefore, from~\eqref{H>0}, we obtain $\vec{v}=0$ which is a contradiction. The proof of~\eqref{coer1} is complete.
\end{proof}

\section{Decomposition around the sum of $N$ solitons}\label{S:3}
We prove in this section a general decomposition result around the sum of $N$ solitons. Let $N\ge 1$ and for any 
$n\in \{1,\ldots,N\}$, let $\bell_{n}=\ell_{n}{\be}_{1} $ where $-1<\ell_{n}<1$ and $\ell_{n}\ne \ell_{n'}$ for $n\ne n'$.
Let $q_{1},\ldots,q_{N}$ be any non-degenerate decaying excited states.
Denote by $I$ and $I^{0}$ the following two sets of indices
\begin{equation*}
\begin{aligned}
I^{0}&=\{(n,k):n=1,\ldots,N,\ k=1,\ldots, K_{n}\}, \quad |I^{0}|=\CARD I^{0}=\sum_{n=1}^{N}K_{n},\\
I&=\{(n,j):n=1,\ldots,N,\ j=1,\ldots, J_{n}\}, \ \ \quad |I|=\CARD I=\sum_{n=1}^{N}J_{n}.
\end{aligned}
\end{equation*}
We denote by $(\lambda_{n,j})_{(n,j)\in I}$, $\left(Y_{(n,j)}\right)_{(n,j)\in I}$, 
$\left(\Psi_{(n,k)}\right)_{(n,k)\in I^{0}}$ the negative eigenvalues, corresponding eigenfunctions and kernel functions for $q_{n}$ as defined in Lemma 2.1.

For $n=1,\ldots,N$, set
\begin{equation*}
Q_{n}(t,x)=q_{n,\ell_{n}}(x-\bell_{n}t),\quad 
\vec{Q}_{n}(t,x)=\left(\begin{array}{c}
	Q_{n}(t,x)\\
	-\ell_{n}\partial_{x_{1}}Q_{n}(t,x)
\end{array}\right).
\end{equation*}
Similarly, for $(n,k)\in I^{0}$,
\begin{equation*}
\Psi_{n,k}(t,x)=\Psi_{(n,k),\ell_{n}}(x-\bell_{n}t),\quad 
\vec{\Psi}_{n,k}(t,x)=\vec{\Psi}_{(n,k),\ell_{n}}(x-\bell_{n}t),
\end{equation*}
and for $(n,j)\in I$,
\begin{equation*}
\vec{Y}^{\pm}_{n,j}(t,x)=\vec{Y}^{\pm}_{(n,j),\ell_{n}}(x-\bell_{n}t),\quad 
\vec{Z}^{\pm}_{n,j}=\mathcal{H}_{n}\vec{Y}^{\pm}_{n,j},
\end{equation*}
where
\begin{equation*}
 \mathcal{H}_{n}=\left(\begin{array}{cc}
-\Delta-f'(Q_{n})& -\ell_{n}\partial_{x_{1}}\\
\ell_{n}\partial_{x_{1}}&1
\end{array}\right).
\end{equation*}
 Consider a time dependent $C^{1}$ function $\boldsymbol{b}$ of the form
\begin{equation*}
\boldsymbol{b}=(b_{n,k})_{(n,k)\in I^{0}}\in \R^{|I_{0}|}\quad \mbox{with}\quad |\boldsymbol{b}|\ll 1.
\end{equation*}
We introduce
\begin{equation*}
U=\sum_{n=1}^{N}Q_{n},\quad V=\sum_{(n,k)\in I^{0}}b_{n,k}\Psi_{n,k},
\end{equation*} 
\begin{equation*}
G=f(U+V)-\sum_{n=1}^{N}f(Q_{n})-\sum_{(n,k)\in I^{0}}b_{n,k}f'(Q_{n})\Psi_{n,k}.
\end{equation*}
First, we start with a technical lemma.
\begin{lemma}\label{le:tech}
Let $W_{1}$ and $W_{2}$ be continuous functions such that,
\begin{equation*}
|W_{1}(x)|+|W_{2}(x)|\lesssim \langle x\rangle^{-4}\quad \mbox{on}\  \RR.
\end{equation*}
Define
\begin{equation*}
W_{n_{1}}(t,x)=W_{1}\left(x-\bell_{n_{1}}t\right),\quad 
W_{n_{2}}(t,x)=W_{2}\left(x-\bell_{n_{2}}t\right).
\end{equation*}
Let $0<\alpha_{1}\le \alpha_{2}$ be such that $\alpha_{1}+\alpha_{2}>\frac{5}{4}$. 
There exists $T_{0}\gg 1$ such that,
for all $n_{1},n_{2}\in \{1,\ldots, N\}$ with $n_{1}\ne n_{2}$ and $t\ge T_{0}$, the following hold.
\begin{enumerate}
	\item If $\alpha_{2}> \frac{5}{4}$, 
	\begin{equation}\label{tech1}
	\int_{\RR}\big|W_{n_{1}}\big|^{\alpha_{1}}\big|W_{n_{2}}\big|^{\alpha_{2}}\d x\lesssim t^{-4\alpha_{1}}.
	\end{equation}
	\item  If $\alpha_{2}\le\frac{5}{4}$,
	\begin{equation}\label{tech2}
	\int_{\RR}\big|W_{n_{1}}\big|^{\alpha_{1}}\big|W_{n_{2}}\big|^{\alpha_{2}}\d x\lesssim t^{5-4(\alpha_{1}+\alpha_{2})}.
	\end{equation}
\end{enumerate}
\end{lemma}
\begin{proof}
	For $k=1,2$, we denote
	\begin{equation*}
	\rho_{k}=x-\bell_{n_{k}}t,\quad 
	\Omega_{k}(t)=\left\{x\in \R^{5}:|\rho_{k}|\le 10^{-1}{|\bell_{n_{1}}-\bell_{n_{2}}|t}\right\}.
	\end{equation*}
	Let $T_{0}\gg 1$ large enough. For $t\ge T_{0}$, from the decay properties of $W_{1}$ and $W_{2}$,
	\begin{equation}\label{leest1}
	\left|W_{n_2}(t,x)\right|
	\lesssim \langle \rho_{2}\rangle^{-4}
	\lesssim (\langle \rho_{1} \rangle+t)^{-4},\quad \mbox{for}\ x\in \Omega_{1},
	\end{equation}
		\begin{equation}\label{leest2}
	\left|W_{n_1}(t,x)\right|
	\lesssim \langle \rho_{1}\rangle^{-4}
	\lesssim (\langle \rho_{2} \rangle+t)^{-4},\quad \mbox{for}\ x\in \Omega_{2},
	\end{equation}
	\begin{equation}\label{leest3}
		\left|W_{n_1}(t,x)\right|\lesssim \left(\langle \rho_{1}\rangle+t\right)^{-4}\lesssim t^{-4},\quad \quad \ \mbox{for}\ x\in \Omega_{1}^{C},
	\end{equation}
	\begin{equation}\label{leest4}
	\left|W_{n_2}(t,x)\right|\lesssim \left(\langle \rho_{2}\rangle+t\right)^{-4}\lesssim t^{-4},\quad \quad \ \mbox{for}\ x\in \Omega_{2}^{C}.
	\end{equation}
	Proof of (i). \emph{Case $\alpha_{1}> \frac{5}{4},\alpha_{2}> \frac{5}{4}$}. From~\eqref{leest1} and~\eqref{leest3}, we obtain
	\begin{equation*}
	\int_{\Omega_{1}}|W_{n_1}|^{\alpha_{1}}|W_{n_{2}}|^{\alpha_{2}}\d x\lesssim t^{-4\alpha_{2}}\int_{\Omega_{1}}|W_{n_1}|^{\alpha_{1}}\d x\lesssim t^{-4\alpha_{2}},
	\end{equation*}
	\begin{equation*}
    \int_{\Omega^{C}_{1}}|W_{n_1}|^{\alpha_{1}}|W_{n_{2}}|^{\alpha_{2}}\d x\lesssim t^{-4\alpha_{1}}\int_{\Omega^{C}_{1}}|W_{n_2}|^{\alpha_2}\d x\lesssim t^{-4\alpha_{1}},
	\end{equation*}
	which implies~\eqref{tech1}.
	
	\emph{Case $0<\alpha_{1}\le\frac{5}{4},\alpha_{2}>\frac{5}{4}$.} By~\eqref{leest1} and change of variable,
	\begin{equation*}
	\begin{aligned}
	\int_{\Omega_{1}}|W_{n_1}|^{\alpha_{1}}|W_{n_{2}}|^{\alpha_{2}}\d x
	&\lesssim
	\int_{\R^{5}} \langle \rho_{1}\rangle^{-4\alpha_1}(\langle \rho_{1}\rangle+t)^{-4\alpha_2}\d x\\
	&\lesssim \int_{\R^{5}}\langle x\rangle^{-4\alpha_1}\left(\langle x \rangle+t\right)^{-4\alpha_2}\d x\\
	&\lesssim t^{5-4(\alpha_{1}+\alpha_2)}\lesssim
	t^{-4\alpha_{1}}.
	\end{aligned}
	\end{equation*}
	Using again~\eqref{leest3},
	\begin{equation*}
	\int_{\Omega_{1}^{C}}|W_{n_1}|^{\alpha_{1}}|W_{n_2}|^{\alpha_{2}}\d x\lesssim
	t^{-4\alpha_{1}}\int_{\Omega_{1}^{C}}|W_{n_2}|^{\alpha_{2}}\d x\lesssim
	t^{-4\alpha_{1}}.
	\end{equation*}
	These estimates imply~\eqref{tech1}.
	
	Proof of (ii). First, from~\eqref{leest1},~\eqref{leest2} and change of variable, as before,
	\begin{equation*}
	\begin{aligned}
	&\int_{\Omega_{1}}|W_{n_1}|^{\alpha_{1}}|W_{n_2}|^{\alpha_{2}}\d x
	+\int_{\Omega_{2}}|W_{n_1}|^{\alpha_{1}}|W_{n_2}|^{\alpha_{2}}\d x\\
	&\lesssim\int_{\R^{5}}
	\left(\langle x\rangle^{-4\alpha_1}\left(\langle x \rangle+t\right)^{-4\alpha_2}+\langle x\rangle^{-4\alpha_2}\left(\langle x \rangle+t\right)^{-4\alpha_1}\right)\d x\lesssim t^{5-4(\alpha_{1}+\alpha_2)}.
	\end{aligned}
	\end{equation*}
	Second, by~\eqref{leest3},~\eqref{leest4}, the H\"older's inequality and change of variable,
	\begin{equation*}
	\begin{aligned}
	\int_{\left(\Omega_{1}\cup \Omega_{2}\right)^{C}}|W_{n_1}|^{\alpha_{1}}|W_{n_2}|^{\alpha_{2}}\d x
	&\lesssim \left(\int_{\Omega_{1}^{C}}|W_{n_1}|^{\alpha_{1}+\alpha_{2}}\right)^{\frac{\alpha_{1}}{\alpha_{1}+\alpha_{2}}}\left(\int_{\Omega_{2}^{C}}|W_{n_2}|^{\alpha_{1}+\alpha_{2}}\right)^{\frac{\alpha_{2}}{\alpha_{1}+\alpha_{2}}}\\
	&\lesssim \int_{\R^{5}}\left(\langle x \rangle+t\right)^{-4(\alpha_{1}+\alpha_{2})}\d x
	\lesssim t^{5-4(\alpha_{1}+\alpha_{2})}.
	\end{aligned}
	\end{equation*}
	Gathering above estimates, we obtain~\eqref{tech2}.
\end{proof}

Second, we introduce some pointwise estimates following from Taylor expansion, and omit its proof.
\begin{lemma} The following estimates hold.
	\begin{enumerate}
		\item For all $n=1,\ldots,N$,
		\begin{equation}\label{Taylor3}
		\big|f'(U)-f'(Q_{n})\big|\lesssim \sum_{n'\ne n}\left(|Q_{n'}||Q_{n}|^{\frac{1}{3}}+|Q_{n'}|^{\frac{4}{3}}\right).
		\end{equation}
		\item We have
		\begin{equation}\label{Taylor2}
		\left|f'(U+V)-f'(U)\right|\lesssim 
		\sum_{n=1}^{N}\left|V\right||Q_{n}|^{\frac{1}{3}}+\left|V\right|^{\frac{4}{3}}.
		\end{equation}
		\item For all $s\in \R$,  
		\begin{equation}\label{Taylor1}
		\left|f(U+V+s)-f(U+V)-f'(U+V)s\right|\lesssim
		\left(|U|^{\frac{1}{3}}+\left|V\right|^{\frac{1}{3}}\right)|s|^{2}+|s|^{\frac{7}{3}},
		\end{equation}
		\begin{equation}\label{Taylor4}
		\left|F(U+V+s)-F(U+V)-f(U+V)s\right|\lesssim
		\left(|U|^{\frac{4}{3}}+\left|V\right|^{\frac{4}{3}}\right)|s|^{2}+|s|^{\frac{10}{3}}.
		\end{equation}
	\end{enumerate}
\end{lemma}

Third, we state some preliminary estimates related to the nonlinear interaction. We decompose $G$ as 
\begin{equation*}
G=G_{1}+G_{2}=G_{1,1}+G_{1,2}+G_{1,3}+G_{2},
\end{equation*}
where
\begin{equation*}
\begin{aligned}
&G_{1,1}=f(U+V)-f(U)-f'(U)V-\frac{1}{2}f''(U)V^{2},\\
&G_{1,2}=f'(U)V-\sum_{(n,k)\in I^{0}}b_{n,k}f'(Q_{n})\Psi_{n,k},\\
&G_{1,3}=f(U)-\sum_{n=1}^{N}f(Q_{n})\quad \mbox{and}\quad G_{2}=\frac{1}{2}f''(U)V^{2}.
\end{aligned}
\end{equation*}
For $n=1,\ldots,N$, set 
\begin{equation*}
G_{3,n}=\frac{1}{2}\sum_{k,k'=1}^{K_{n}}b_{n,k}b_{n,k'}f''(Q_{n})\Psi_{n,k}\Psi_{n,k'}\quad 
\mbox{and}\quad G_{3}=\sum_{n=1}^{N}G_{3,n}.
\end{equation*}
\begin{lemma}
	There exists $T_{0}\gg 1$ such that the following estimates hold.
	\begin{enumerate}
		\item \emph{Estimates on $G_{1}$.} For $t\ge T_{0}$,
	\begin{align}\label{est:G123}
	\|G_{1}\|_{L^{2}}\lesssim\|G_{1,1}\|_{L^{2}}+\|G_{1,2}\|_{L^{2}}+\|G_{1,3}\|_{L^{2}}\lesssim \sum_{(n,k)\in  I^{0}}|b_{n,k}|^{\frac{7}{3}}+t^{-4}. 
	\end{align}
	\item \emph{Expansion of $G_{2}$.} For $t\ge T_{0}$,
	\begin{equation}\label{est:G4}
	\begin{aligned}
	&\left\|G_{2}-G_{3}\right\|_{L^{2}}\lesssim \sum_{(n,k)\in I^{0}}|b_{n,k}|^{\frac{7}{3}}+t^{-4}.
	\end{aligned}
	\end{equation}
	\item \emph{Estimate on $G$.} For $t\ge T_{0}$,
	\begin{equation}\label{est:G}
	\|G\|_{L^{2}}\lesssim \sum_{(n,k)\in I^{0}}|b_{n,k}|^{2}+t^{-4}.
	\end{equation}
\end{enumerate}
\end{lemma}
\begin{proof}
	Proof of (i). From Taylor formula, we have
	\begin{equation*}
	\big|G_{1,1}\big|\lesssim \big|V\big|^{\frac{7}{3}}\lesssim \sum_{(n,k)\in I^{0}}|b_{n,k}|^{\frac{7}{3}}\big|\Psi_{n,k}\big|^{\frac{7}{3}},\quad \big|G_{1,3}\big|\lesssim \sum_{n'\ne n}|Q_{n}|^{\frac{4}{3}}|Q_{n'}|,
	\end{equation*}
	\begin{equation*}
	\big|G_{1,2}\big|\lesssim \sum_{(n,k)\in I^{0}}|b_{n,k}|\bigg[\sum_{n'\ne n}|Q_{n}|^{\frac{1}{3}}|Q_{n'}||\Psi_{n,k}|+\sum_{n'\ne n }|Q_{n'}|^{\frac{4}{3}}|\Psi_{n,k}|\bigg].
	\end{equation*}
	Based on above pointwise estimates, Lemma~\ref{le:tech} (i), the Cauchy-Schwarz inequality and the Young's inequality, we obtain~\eqref{est:G123}.
	
	Proof of (ii). Observe that
	\begin{equation*}
	V^{2}=\sum_{n=1}^{N}\sum_{k,k'=1}^{K_{n}}b_{n,k}b_{n,k'}\Psi_{n,k}\Psi_{n,k'}+\sum_{n\ne n'}\sum_{k=1}^{K_{n}}\sum_{k'=1}^{K_{n'}}b_{n,k}b_{n,k'}\Psi_{n,k}\Psi_{n',k'}.
	\end{equation*} 
	Thus, we can decompose
	\begin{equation*}
	G_{2}=G_{3}+G_{2,1}+G_{2,2},
	\end{equation*}
	where
	\begin{equation*}
	\begin{aligned}
	G_{2,1}&=\frac{1}{2}\sum_{n\ne n'}\sum_{k=1}^{K_{n}}\sum_{k'=1}^{K_{n'}}b_{n,k}b_{n',k'}f''(U)\Psi_{n,k}\Psi_{n',k'},\\
	G_{2,2}&=\frac{1}{2}\sum_{n=1}^{N}\sum_{k,k'=1}^{K_{n}}b_{n,k}b_{n,k'}\big(f''(U)-f''(Q_{n})\big)\Psi_{n,k}\Psi_{n,k'}.
	\end{aligned}
	\end{equation*}
	By (i) of Lemma~\ref{le:tech}, we have
	\begin{equation*}
	\begin{aligned}
	\|G_{2,1}\|_{L^{2}}
	&\lesssim \sum_{n\ne n'}\sum_{k=1}^{K_{n}}\sum_{k'=1}^{K_{n'}}|b_{n,k}b_{n',k'}|\||Q_{n}|^{\frac{1}{3}}\Psi_{n,k}\Psi_{n',k'}\|_{L^{2}}\\
	&+\sum_{n\ne n'}\sum_{k'=1}^{K_{n'}}b^{2}_{n',k'}\||Q_{n}|^{\frac{1}{3}}\Psi^{2}_{n',k'}\|_{L^{2}}\lesssim
	\sum_{(n,k)\in I^{0}}|b_{n,k}|^{\frac{7}{3}}+t^{-4}.
	\end{aligned}
	\end{equation*}
	From Taylor formula, 
	\begin{equation*}
	\left|f''(U)-f''(Q_{n})\right|\lesssim \sum_{n'\ne n}|Q_{n'}|^{\frac{1}{3}}\quad \mbox{for all}\ n=1,\ldots,N.
	\end{equation*}
	Thus, using again (i) of Lemma~\ref{le:tech},
	\begin{equation*}
	\|G_{2,2}\|_{L^{2}}\lesssim \sum_{(n,k)\in I^{0}}|b_{n,k}|^{2}\bigg(\sum_{n'\ne n}\||Q_{n'}|^{\frac{1}{3}}\Psi_{n,k}^{2}\|_{L^{2}}\bigg)\lesssim \sum_{(n,k)\in I^{0}}|b_{n,k}|^{\frac{7}{3}}+t^{-4}.
	\end{equation*}
	Gathering above estimates, we obtain~\eqref{est:G4}.
	
	Proof of (iii). Estimate~\eqref{est:G} is a consequence of~\eqref{est:G123} and~\eqref{est:G4}.
\end{proof}

Last we prove a standard decomposition result around the sum of $N$ solitons.
\begin{proposition}[Properties of the decomposition]\label{le:deco}
There exist $T_{0}\gg 1$ and $0<\delta_{0}\ll 1$ such that if $\vec{u}(t)=(u(t),\partial_{t}u(t))$
is a solution of~\eqref{wave} on $[T_{1},T_{2}]$, where $T_{0}\le T_{1}<T_{2}<+\infty$, such that 
for any $t\in [T_{1},T_{2}]$
\begin{equation}\label{estQ0}
\|\vec{u}(t)-\sum_{n=1}^{N}\vec{Q}_{n}(t)\|_{\E}\le \delta_{0},
\end{equation}
then there exists $C^{1}$ functions $\bb=(b_{n,k})_{(n,k)\in I^{0}}$ on $[T_{1},T_{2}]$ such that,
$\vec{\varepsilon}(t)$ being defined by 
\begin{equation}\label{defe}
\vec{\varepsilon}(t)=
\left(\begin{array}{c}\varepsilon \\\eta \end{array}\right)=\vec{u}(t)-\sum_{n=1}^{N}\vec{Q}_{n}(t)-
\sum_{(n,k)\in I^{0}}b_{n,k}\vec{\Psi}_{n,k}
\end{equation}
the following hold on $[T_{1},T_{2}]$,

\emph{\rm(i)  First properties of the decomposition.}
For all $(n,k)\in I^{0}$ and $t\in [T_{1},T_{2}]$,
\begin{equation}\label{orth}
\big(\vec{\varepsilon}(t), \vec{\Psi}_{n,k}(t))_{\E}=0,\quad |b_{n,k}(t)|\lesssim \|\vec{u}(t)-\sum_{n=1}^{N}\vec{Q}_{n}(t)\|_{\E}.
\end{equation}

\emph{\rm(ii)  Equation of $\vec{\varepsilon}$.} It holds
\begin{equation}\label{equpe}
\left\{ \begin{aligned}
&\partial_t\varepsilon=\eta-{\rm{Mod}_{1}},\\
&\partial_{t}\eta=\Delta\varepsilon+f(U+V+\varepsilon)-f(U+V)
-{\rm{Mod}_{2}}
+G
\end{aligned}\right.
\end{equation}
where
\begin{equation*}
{\rm{Mod}_{1}}=\sum_{(n,k)\in I^{0}}\dot{b}_{n,k}{\Psi}_{n,k}\quad \mbox{and}\quad  {\rm{Mod}}_{2}=-\sum_{(n,k)\in I^{0}}\dot{b}_{n,k}\ell_{n}\partial_{x_{1}}\Psi_{n,k}.
\end{equation*}

\emph{\rm(iii) Estiamtes for $\bb$.} For $t\in [T_{1},T_{2}]$, we have
\begin{equation}\label{estb}
\sum_{(n,k)\in I^{0}}|\dot{b}_{n,k}(t)|\lesssim \|\vec{\varepsilon}(t)\|_{\E}+\sum_{(n,k)\in I^{0}}|b_{n,k}|^{2}+t^{-4}.
\end{equation}

\emph{\rm(iv) Unstable directions.} Let $a^{\pm}_{n,j}=\left(\vec{\varepsilon},\vec{Z}^{\pm}_{n,j}\right)_{L^{2}}$ for all $(n,j)\in I$. Then
\begin{equation}\label{equa+-}
\left|\frac{{\rm d}}{{\rm{d}}t}a^{\pm}_{n,j}(t)\pm \alpha_{n,j}a^{\pm}_{n,j}(t)\right|\lesssim 
\|\vec{\varepsilon}(t)\|^{2}_{\E}+\sum_{(n,k)\in I^{0}}|b_{n,k}|^{2}+t^{-4}.
\end{equation}
where $\alpha_{n,j}=\lambda_{n,j}(1-\ell^{2}_{n})^{^{\frac{1}{2}}}>0$.
\end{proposition}
\begin{proof}
	Proof of (i).
Let $T_{0}\gg 1$, fix $t\ge T_{0}$. Note that, the orthogonality conditions in~\eqref{orth}
is equivalent to the following matrix identity,
\begin{equation*}
\mathcal{M}\bb=\left(\bigg(\vec{u}-\sum_{n=1}^{N}\vec{Q}_{n},\vec{\Psi}_{n,k}\bigg)_{\E}\right)_{(n,k)\in I^{0}},
\end{equation*}
where $\bb=(b_{n,k})_{(n,k)\in I^{0}}$ written in one row and 
\begin{equation*}
\mathcal{M}=\left(\bigg(\vec{\Psi}_{n,k},\vec{\Psi}_{n',k'}\bigg)_{\E}\right)_{(n,k),(n',k')\in I^{0}}.
\end{equation*}
Moreover, from (ii) of Lemma~\ref{le:tech} and the decay property of $\vec{\Psi}_{n,k}$,
\begin{equation*}
\mathcal{M}={\rm{diag}}\left(\mathcal{M}_{1},\ldots,\mathcal{M}_{N}\right)+O(t^{-2})\quad \mbox{where}\
\mathcal{M}_{n}=\left((\vec{\Psi}_{n,k},\vec{\Psi}_{n,k'})_{\E}\right)_{k,k'=1,\ldots,K_{n}}.
\end{equation*}
Note that for fixed $n$, the family $\left(\vec{\Psi}_{n,k}\right)_{k=1,\ldots,K_{n}}$ being linearly independent,
the Gram matrix $\mathcal{M}_{n}$ is invertible.
Hence, $\mathcal{M}$ is invertible, for $T_{0}$ large enough, and $\mathcal{M}^{-1}$ has uniform norm in $t\ge T_{0}$.
Therefore, we obtain 
\begin{equation*}
\bb=\mathcal{M}^{-1}\bigg(\bigg(\vec{u}-\sum_{n=1}^{N}\vec{Q}_{n},\vec{\Psi}_{n,k}\bigg)_{\E}\bigg)_{(n,k)\in I^{0}}
\end{equation*}
and
\begin{equation*}
|\bb|\le\left\|\mathcal{M}^{-1}\right\|\bigg\|\bigg(\bigg(\vec{u}-\sum_{n=1}^{N}\vec{Q}_{n},\vec{\Psi}_{n,k}\bigg)_{\E}\bigg)_{(n,k)\in I^{0}}\bigg\|\lesssim \bigg\|\vec{u}(t)-\sum_{n=1}^{N}\vec{Q}_{n}(t)\bigg\|_{\E}.
\end{equation*}

Proof of (ii). First, by the definition of $\vec{\varepsilon}$ in~\eqref{defe}
\begin{align*}
\partial_{t}\varepsilon
&=\partial_{t}u-\sum_{n=1}^{N}\partial_{t}Q_{n}
-\sum_{(n,k)\in I^{0}}b_{n,k}\partial_{t}\Psi_{n,k}-\sum_{(n,k)\in I^{0}}\dot{b}_{n,k}\Psi_{n,k}\\
&=\eta-\sum_{(n,k)\in I^{0}}\dot{b}_{n,k}\Psi_{n,k}.
\end{align*}
Second, using again~\eqref{defe} and~\eqref{wave},
\begin{align*}
\partial_{t}\eta&=\partial_{t}^{2}u+\sum_{n=1}^{N}\partial_{t}\big(\ell_{n}\partial_{x_{1}}Q_{n})
+\sum_{(n,k)\in I^{0}}b_{n,k}\partial_{t}(\ell_{n}\partial_{x_{1}}\Psi_{n,k})
+\sum_{(n,k)\in I^{0}}\dot{b}_{n,k}\big(\ell_{n}\partial_{x_{1}}\Psi_{n,k}\big)\\
&=\Delta u+f(u)-\sum_{n=1}^{N}\ell_{n}^{2}\partial_{x_{1}}^{2}Q_{n}
-\sum_{(n,k)\in I^{0}}b_{n,k}\ell^{2}_{n}\partial_{x_{1}}^{2}\Psi_{n,k}
+\sum_{(n,k)\in I^{0}}\dot{b}_{n,k}(\ell_{n}\partial_{x_{1}}\Psi_{n,k}).
\end{align*}
From~\eqref{defe}, $-(1-\ell_{n}^{2})\partial_{x_{1}}^{2}Q_{n}-\bar{\Delta} Q_{n}-f(Q_{n})=0$, $-(1-\ell_{n}^{2})\partial_{x_{1}}^{2}\Psi_{n,k}-\bar{\Delta} \Psi_{n,k}-f'(Q_{n})\Psi_{n,k}=0$ and the definition of $G$,
\begin{align*}
&\Delta u+f(u)-\sum_{n=1}^{N}\ell_{n}^{2}\partial_{x_{1}}^{2}Q_{n}
-\sum_{(n,k)\in I^{0}}b_{n,k}\ell^{2}_{n}\partial_{x_{1}}^{2}\Psi_{n,k}\\
&=\Delta\varepsilon+f(U+V+\varepsilon)-f(U+V)+G.
\end{align*}
Therefore,
\begin{equation*}
\partial_{t}\eta=\Delta \varepsilon+f(U+V+\varepsilon)-f(U+V)+\sum_{(n,k)\in I^{0}}\dot{b}_{n,k}\ell_{n}\partial_{x_{1}}\Psi_{n,k}+G.
\end{equation*}

Proof of (iii). First, we decompose
\begin{equation*}
f(U+V+\varepsilon)-f(U+V)=\sum_{n=1}^{N}f'(Q_{n})\varepsilon+R_{1}+R_{2}+R_{3},
\end{equation*}
where
\begin{equation*}
R_{1}=f(U+V+\varepsilon)-f(U+V)-f'(U+V)\varepsilon,
\end{equation*}
\begin{equation*}
R_{2}=\left(f'(U+V)-f'(U)\right)\varepsilon,\quad 
R_{3}=\left(f'(U)-\sum_{n=1}^{N}f'(Q_{n})\right)\varepsilon.
\end{equation*}
Therefore, from~\eqref{equpe}, we obtain
\begin{equation}\label{eqpvece}
\partial_{t} {\vec{\varepsilon}}=\vec{\mathcal{L}}\vec{\varepsilon}-\vec{\rm{{Mod}}}+\vec{G}+\vec{R_{1}}+\vec{R}_{2}+\vec{R}_{3},
\end{equation}
where
\begin{equation*}
\vec{\mathcal{L}}=\left(\begin{array}{cc}0 & 1 \\ \Delta+\sum_{n=1}^{N}f'(Q_{n}) & 0 \end{array}\right),\
\vec{{\rm{Mod}}}=\left(\begin{array}{c}
{\rm{Mod}}_{1}\\
{{\rm{Mod}}}_{2}
\end{array}\right),\
\end{equation*}
and
\begin{equation*}
\vec{G}=\left(\begin{array}{c} 0\\ G \end{array}\right),\
\vec{R}_{1}=\left(\begin{array}{c} 0\\ R_{1} \end{array}\right),\
\vec{R}_{2}=\left(\begin{array}{c} 0\\ R_{2} \end{array}\right),\
\vec{R}_{3}=\left(\begin{array}{c} 0\\ R_{3}\end{array}\right).
\end{equation*}
We differentiate the orthogonality condition $(\vec{\varepsilon},\vec{\Psi}_{n,k})_{\E}=0$ in~\eqref{orth} and using~\eqref{eqpvece},
\begin{align*}
0=\frac{\rm{d}}{{\rm d}t}(\vec{\varepsilon},\vec{\Psi}_{n,k})_{\E}
=&(\partial_{t}\vec{\varepsilon},\vec{\Psi}_{n,k})_{\E}
-(\vec{\varepsilon},\partial_{t}\vec{\Psi}_{n,k})_{\E}\\
=&(\mathcal{\vec{L}}\vec{\varepsilon}, \vec{\Psi}_{n,k})_{\E}-\sum_{(n',k')\in I^{0}}\dot{b}_{n',k'}(\vec{\Psi}_{n,k},\vec{\Psi}_{n',k'})_{\E}
+\left(\vec{G},\vec{\Psi}_{n,k}\right)_{\E}\\
+&\left(\vec{R}_{1},\vec{\Psi}_{n,k}\right)_{\E}
+\left(\vec{R}_{2}+\vec{R}_{3},\vec{\Psi}_{n,k}\right)_{\E}
-\left(\vec{\varepsilon},\ell_{n}\partial_{x_{1}}\vec{\Psi}_{n,k}\right)_{\E}.
\end{align*}
By integration by parts and the decay properties of $\vec{\Psi}_{n,k}$, the first term is 
\begin{equation*}
(\vec{\mathcal{L}}\vec{\varepsilon}, \vec{\Psi}_{n,k})_{\E}
=\left({\rm{J}}\vec{\varepsilon},\vec{\mathcal{L}}{\rm{J}}\vec{\Psi}_{n,k}\right)_{\E}=
O\left(\|\vec{\varepsilon}\|_{\E}\right).
\end{equation*}
From~\eqref{est:G} and the Cauchy-Schwarz inequality,
\begin{equation*}
\left|\left(\vec{G},\vec{\Psi}_{n,k}\right)_{\E}\right|\lesssim \|G\|_{L^{2}}\lesssim \sum_{(n,k)\in I^{0}}|b_{n,k}|^{2}+t^{-4}.
\end{equation*}
Next, from~\eqref{Taylor3},~\eqref{Taylor2},~\eqref{Taylor1}, the Sobolev embedding theorem and the decay properties of $\vec{\Psi}_{n,k}$,
\begin{equation*}
\left|\left(\vec{R}_{1},\vec{\Psi}_{n,k}\right)_{\E}\right|
\lesssim \left|\left((|U|^{\frac{1}{3}}+|V|^{\frac{1}{3}})\varepsilon^{2}+|\varepsilon|^{\frac{7}{3}},
\partial_{x_{1}}\Psi_{n,k}\right)_{L^{2}}\right|
\lesssim \|\vec{\varepsilon}\|_{\E}^{2},
\end{equation*}
\begin{equation*}
\left|\left(\vec{R}_{2}+\vec{R}_{3},\vec{\Psi}_{n,k}\right)_{\E}\right|+
\left|\left(\vec{\varepsilon},\ell_{n}\partial_{x_{1}}\vec{\Psi}_{n,k}\right)_{\E}\right|\lesssim
\|R_{2}\|_{L^{2}}+\|R_{3}\|_{L^{2}}+\|\vec{\varepsilon}\|_{\E}
\lesssim \|\vec{\varepsilon}\|_{\E}.
\end{equation*}
Gathering above estimates and proceeding similarly for all $(n,k)\in I^{0}$,
\begin{equation*}
\mathcal{M}\dot{\bb}=O\bigg(\|\vec{\varepsilon}\|_{\E}+\sum_{(n,k)\in I^{0}}|b_{n,k}|^{2}+t^{-4}\bigg),
\end{equation*}
where $\dot{\bb}=\left(\dot{b}_{n,k}\right)_{(n,k)\in I^{0}}$ written in one row. Therefore, from the matrix
$\mathcal{M}^{-1}$ being uniformly bounded, we obtain~\eqref{estb}.

Proof of (iv). Using~\eqref{eqpvece}, we compute
\begin{equation*}
\begin{aligned}
\frac{\d}{\d t}a_{n,j}^{\pm}
=&\left(\partial_{t}\vec{\varepsilon},\vec{Z}_{n,j}^{\pm}\right)_{L^{2}}
+\left(\vec{\varepsilon},\partial_{t}\vec{Z}_{n,j}^{\pm}\right)_{L^{2}}\\
=&\left(\vec{\mathcal{L}}\vec{\varepsilon},\vec{Z}_{n,j}^{\pm}\right)_{L^{2}}
-\ell_{n}\left(\vec{\varepsilon},\partial_{x_{1}}\vec{Z}_{n,j}^{\pm}\right)_{L^{2}}
-\left(\vec{\rm{Mod}},\vec{Z}_{n,j}^{\pm}\right)_{L^{2}}\\
+&\left(\vec{G},\vec{Z}^{\pm}_{n,j}\right)_{L^{2}}+\left(\vec{R}_{1},\vec{Z}^{\pm}_{n,j}\right)_{L^{2}}
+\left(\vec{R}_{2},\vec{Z}^{\pm}_{n,j}\right)_{L^{2}}+\left(\vec{R}_{3},\vec{Z}^{\pm}_{n,j}\right)_{L^{2}}.
\end{aligned}
\end{equation*}
From~\eqref{HY}, integration by parts and $-{\rm{J}}^{2}$ is identity matrix, we have
\begin{equation*}
\begin{aligned}
\left(\vec{\mathcal{L}}\vec{\varepsilon},\vec{Z}_{n,j}^{\pm}\right)_{L^{2}}
-\ell_{n}\left(\vec{\varepsilon},\partial_{x_{1}}\vec{Z}_{n,j}^{\pm}\right)_{L^{2}}
&=-\left(\vec{\varepsilon},\mathcal{H}_{n}{\rm{J}}\vec{Z}_{n,j}^{\pm}\right)_{L^{2}}+
\sum_{n'\ne n}\left(\varepsilon,f'(Q_{n'})Z^{\pm}_{n,j}\right)_{L^{2}}\\
&=\mp \alpha_{n,j} a_{n,j}^{\pm}+
\sum_{n'\ne n}\left(\varepsilon,f'(Q_{n'})Z^{\pm}_{n,j}\right)_{L^{2}},
\end{aligned}
\end{equation*}
where 
\begin{equation*}
Z^{\pm}_{n,j}=\pm\alpha_{n,j}\left(Y_{(n,j),\ell_{n}}e^{\mp\frac{\ell_{n}\lambda_{n,j}}{\sqrt{1-\ell^{2}_{n}}}x_{1}}\right)\left(\cdot-\bell_{n}t\right).
\end{equation*}
By Sobolev embedding theorem and (i) of Lemma~\ref{le:tech},
\begin{equation*}
\sum_{n'\ne n}\left|\left(\varepsilon,f'(Q_{n'})Z^{\pm}_{n,j}\right)_{L^{2}}\right|
\lesssim \sum_{n'\ne n}\|\varepsilon\|_{L^{\frac{10}{3}}}\left\|f'(Q_{n'})Z^{\pm}_{n,j}\right\|_{L^{\frac{10}{7}}}
\lesssim \|\vec{\varepsilon}\|_{\E}^{2}+t^{-4}.
\end{equation*}
Note that, for all $(n,k)\in I^{0}$ and $(n,j)\in I$, we have
\begin{equation*}
\left(\vec{\Psi}_{n,k},\mathcal{H}_{n}\vec{Y}_{n,j}^{\pm}\right)_{L^{2}}
=\left(\mathcal{H}_{n}\vec{\Psi}_{n,k},\vec{Y}_{n,j}^{\pm}\right)_{L^{2}}=0.
\end{equation*}
Therefore, from (i) of Lemma~\ref{le:tech},~\eqref{estb} and concerning the term with $\vec{\rm{Mod}}$,
\begin{equation*}
\left(\vec{{\rm{Mod}}},\vec{Z}_{n,j}^{\pm}\right)_{L^{2}}
=-\sum_{n'\ne n}\sum_{k=1}^{K_{n'}}\dot{b}_{n',k}
\left(\vec{\Psi}_{n',k},\vec{Z}_{n,j}^{\pm}\right)_{L^{2}}=
O\bigg(\|\vec{\varepsilon}\|_{\E}^{2}+\sum_{(n,k)\in I^{0}}|b_{n,k}|^{2}+t^{-4}\bigg).
\end{equation*}
Next, from~\eqref{est:G} and the Cauchy Schwarz inequality,
\begin{equation*}
\left|\left(\vec{G},\vec{Z}_{n,j}^{\pm}\right)_{L^{2}}\right|\lesssim \|G\|_{L^{2}}\lesssim 
\sum_{(n,k)\in I^{0}}|b_{n,k}|^{2}+t^{-4}.
\end{equation*}
Last, from~\eqref{Taylor3},~\eqref{Taylor2},~\eqref{Taylor1}, (i) of Lemma~\ref{le:tech} and decay properties of ${Z}_{n,j}^{\pm}$,
\begin{equation*}
\left|\left(\vec{R}_{1},\vec{Z}_{n,j}^{\pm}\right)_{L^{2}}\right|\lesssim
 \left|\left(\big(|U|^{\frac{1}{3}}+|V|^{\frac{1}{3}}\big)|\varepsilon|^{2}+|\varepsilon|^{\frac{7}{3}},Z_{n,j}^{\pm}\right)_{L^{2}}\right|\lesssim \|\vec{\varepsilon}\|_{\E}^{2},
\end{equation*}
\begin{equation*}
\left|\left(\vec{R}_{2},\vec{Z}_{n,j}^{\pm}\right)_{L^{2}}\right|\lesssim 
\|\varepsilon\|_{L^{\frac{10}{3}}}\|f'(U+V)-f'(U)\|_{L^{2}}\lesssim \|\vec{\varepsilon}\|_{\E}^{2}+\sum_{(n,k)\in I^{0}}|b_{n,k}|^{2},
\end{equation*}
\begin{equation*}
\left|\left(\vec{R}_{3},\vec{Z}_{n,j}^{\pm}\right)_{L^{2}}\right| \lesssim 
\|\varepsilon\|_{L^{\frac{10}{3}}}\bigg\|\sum_{n'\ne n}\left(|Q_{n'}||Q_{n}|^{\frac{1}{3}}+|Q_{n'}|^{\frac{4}{3}}\right)Z_{n,j}^{\pm}\bigg\|_{L^{\frac{10}{7}}}
\lesssim \|\vec{\varepsilon}\|_{\E}^{2}+t^{-4}.
\end{equation*}
Gathering above estimates and proceeding similarly for all $(n,j)\in I$, we obtain~\eqref{equa+-}.
\end{proof}

\section{Proof of Theorem \ref{main:thm}}\label{S:4}
In this section, we prove the existence of a solution $\vec{u}(t)$ of~\eqref{wave} satisfying~\eqref{thm:est} in
Theorem~\ref{main:thm}. We argue by compactness and obtain $\vec{u}(t)$ as the limit of suitable approximate multi-solitons $\vec{u}_m(t)$.

We start with a technical lemma which constructs well-prepared initial data at $t=T\gg 1$ with sufficient freedom related to unstable directions.
\begin{lemma}\label{le:ini}
	Let $T\gg 1$ and $C\gg1$. For any $\boldsymbol{a}^{}=(a^{}_{n,j})_{(n,j)\in I}\in \R^{|I|}$, there exists 
	\begin{equation*}
	A=(\tilde{a}_{n,j})_{(n,j)\in I}\in \R^{|I|},\quad 
	B=(\tilde{b}_{n,k})_{(n,k)\in I^{0}}\in\R^{|I_{0}|},
	\end{equation*}
	satisfying 
	\begin{equation}\label{est:ab}
	\sum_{(n,j)\in I}|\tilde{a}_{n,j}|+\sum_{(n,k)\in I}|\tilde{b}_{n,k}|
	\le C\sum_{(n,j)\in I}|a_{n,j}|,
	\end{equation}
	such that the function $\vec{\varepsilon}(T)$ defined by 
	\begin{equation}\label{def:eTm}
	\vec{\varepsilon}(T)=\sum_{(n,j)\in I}\tilde{a}_{n,j}(\boldsymbol{a})\vec{Z}^{+}_{n,j}(T)
	 +\sum_{(n,k)\in I^{0}}\tilde{b}_{n,k}(\boldsymbol{a})\vec{\Psi}_{n,k}(T),
	\end{equation}
	satisfies for all $(n,k)\in I^{0}$, $(n,j)\in I$,
	\begin{equation}\label{cond}
	\left(\vec{\varepsilon}(T),\vec{\Psi}_{n,k}(T)\right)_{\E}=0,\quad  \left(\vec{\varepsilon}(T),\vec{Z}^{+}_{n,j}(T)\right)_{L^{2}}=a_{n,j}.
	\end{equation}
	Moreover, the initial data defined by $\vec{u}(T)=\sum_{n=1}^{N}\vec{Q}_{n}(T)+\vec{\varepsilon}(T)$ is modulated in the sense of Proposition~\ref{le:deco} with $b_{n,k}(T)=0$ for all $(n,k)\in I^{0}$ and  $a_{n,j}^{+}(T)=a_{n,j}$ for all $(n,j)\in I$.
\end{lemma}
\begin{proof}
	Our goal is to solve for $A=(\tilde{a}_{n,j})_{(n,j)\in I}$ and $B=(\tilde{b}_{n,k})_{(n,k)\in I^{0}}$ in terms of $\boldsymbol{a}=(a_{n,j})_{(n,j)\in I}$. From the relations, for all $n=1,\ldots,N$,
	\begin{equation*}
	\left(\vec{\Psi}_{n,k},\vec{Z}_{n,j}^{+}\right)_{L^{2}}=
	\left(\mathcal{H}_{n}\vec{\Psi}_{n,k},\vec{Y}_{n,j}^{+}\right)_{L^{2}}=0
	\end{equation*}
	and for all $n\ne n'$,
	\begin{equation*}
	\left(\vec{\Psi}_{n,k}(T),\vec{\Psi}_{n',k'}(T)\right)_{\E}=O(T^{-1}),\quad 
	\left(\vec{\Psi}_{n,k}(T),\vec{Z}_{n',j}^{+}(T)\right)_{\E}=O(T^{-1}),
	\end{equation*}
	\begin{equation*}
	\left(\vec{Z}^{+}_{n,j}(T),\vec{\Psi}_{n',k}^{+}(T)\right)_{L^{2}}=O(T^{-1}),\quad 
	\left(\vec{Z}_{n,j}^{+}(T),\vec{Z}_{n',j}^{+}(T)\right)_{L^{2}}=O(T^{-1}),
	\end{equation*}
the conditions in~\eqref{cond} are equivalent to a linear relation
	between $A$ and $B$ of the following form, for all $(n,k)\in I^{0}$ and $(n,j)\in I$,
	\begin{equation*}
	\begin{aligned}
	&\sum_{j=1}^{J_{n}}\left(\vec{\Psi}_{n,k}(T),\vec{Z}^{+}_{n,j}(T)\right)_{\E}\tilde{a}_{n,j}
	+\sum_{k'=1}^{K_{n}}\left(\vec{\Psi}_{n,k}(T),\vec{\Psi}_{n,k'}(T)\right)_{\E}\tilde{b}_{n,k'}\\
	=&O\bigg(\sum_{n'\ne n}\bigg(\sum_{j'=1}^{J_{n'}}|\tilde{a}_{n',j'}|+\sum_{k'=1}^{K_{n'}}|\tilde{b}_{n',k'}|\bigg)T^{-1}\bigg),
	\end{aligned}
	\end{equation*}
	\begin{equation*}
	\sum_{j'=1}^{J_{n}}\left(\vec{Z}_{n,j'}^{+},\vec{Z}_{n,j}^{+}\right)_{L^{2}}\tilde{a}_{n,j'}=a_{n,j}
	+O\bigg(\sum_{n'\ne n}\bigg(\sum_{j'=1}^{J_{n'}}|\tilde{a}_{n',j'}|+\sum_{k'=1}^{K_{n'}}|\tilde{b}_{n',k'}|\bigg)T^{-1}\bigg).
	\end{equation*}
	 Since the families $\left(\vec{\Psi}_{n,k}\right)_{(n,k)\in I^{0}}$ and $\left(\vec{Z}^{+}_{n,j}\right)_{(n,j)\in I}$ are linear independent,  the above linear system is
	 invertible for $T$ large enough. We obtain the existence and desired properties of $A=(\tilde{a}_{n,j})_{(n,j)\in I}$ and
	 $B=\left(\tilde{b}_{n,k}\right)_{(n,k)\in I^{0}}$ for $T$ large enough.
\end{proof}

To obtain weak convergence on any finite time interval, not just at initial time, we recall a proposition from~\cite{JJ,JeMa} which associates with weak continuity of the flow.

\begin{proposition}\label{prop:weak}
	There exists a constant $\epsilon>0$ such that the following holds.
	Let ${\rm{K}}\subset \dot{H}^{1}\times L^{2}$ be a compact set and let $\vec{u}_{m}:[T_{1},T_{2}]\to \dot{H}^{1}\times L^{2}$ be a sequence of solutions of~\eqref{wave} such that 
	\begin{equation*}
	{\rm{dist}}\left(\vec{u}_{m}(t),{\rm{K}}\right)\le \epsilon,\quad \mbox{for all}\ m\in \mathbb{N}^{+}\ \mbox{and}\ t\in [T_{1},T_{2}].
	\end{equation*}
Suppose that $\vec{u}_{m}(T_{1})\rightharpoonup \vec{u}_{0}$ weakly in $\dot{H}^{1}\times L^{2}$. Then the solution $\vec{u}(t)$ of~\eqref{wave} with the initial data $\vec{u}(T_{1})=\vec{u}_{0}$ is defined for $t\in [T_{1},T_{2}]$ and 
\begin{equation*}
\vec{u}_{m}(T_{1})\rightharpoonup \vec{u}(t),\quad \mbox{weakly in}\ \dot{H}^{1}\times L^{2}\ \mbox{for all}\ t\in [T_{1},T_{2}].
\end{equation*}
\end{proposition}
\begin{proof}
	The proof relies on a standard argument based on the result of profile decomposition stated in~\cite[Proposition 2.8]{DKMAWAVE}. See more details in~\cite[Appendix A.2]{JJ} and~\cite[Appendix]{JeMa}.
	\end{proof}

Let $T_{m}=m$ for all $m\in \mathbb{N}^{+}$. For $\boldsymbol{a}_{m}=(a^{m}_{n,j})_{(n,j)\in I}\in \R^{|I|}$ small to be determined later, we consider 
the solution $\vec{u}_{m}\in C\big(T^{m}_{\rm{max}}; \dot{H}^{1}\times L^{2}\big)$ with the initial data $\vec{u}_{m}(T_{m})$ given by Lemma~\ref{le:ini}. Since $\vec{u}_{m}(T_{m})\in X$, by persistence of regularity (see for instance~\cite[Appendix B]{JJwave5}), we have $\vec{u}_{m}\in C\big(T^{m}_{\rm{max}}; X\big)$. Such regularity will allow energy computations without density argument, see \S\ref{sec:ener}.

\smallskip

The following Proposition is the main part of the proof of Theorem~\ref{main:thm}.
\begin{proposition}[Uniform estimates]\label{main:pro}
	Under the assumptions of Theorem~\ref{main:thm}, there exist $m_{0}\in \mathbb{N}^{+}$ and $T_{0}\gg 1$ 
	such that, for any $m\ge m_{0}$, there exist $\boldsymbol{a}_{m}=(a^{m}_{n,j})_{(n,j)\in I}\in \R^{|I|}$ such that the solution $\vec{u}_{m}$ of~\eqref{wave} with initial data $\vec{u}_{m}(T_{m})$ given by Lemma~\ref{le:ini} is well-defined in $\dot{H}^{1}\times L^{2}$ on the time interval $[T_{0},T_{m}]$
	and satisfies
	\begin{equation}\label{uniest}
	\forall t\in[T_{0},T_{m}],\quad \bigg\|\vec{u}_{m}(t)-\sum_{n=1}^{N}\vec{Q}_{n}(t)\bigg\|_{\E}\le t^{-\frac{8}{5}}.
	\end{equation}
\end{proposition}
\subsection{Proof of Theorem~\ref{main:thm} from Proposition~\ref{main:pro}} \label{SS:mainthm}
We follow the strategy by uniform estimates and compactness introduced in~\cite{JeMa,Ma,MMnls}. Consider the solution $\vec{u}_{m}(t)$ given by Proposition~\ref{main:pro}. On the interval $[T_{0},T_{m}]$, from~\eqref{uniest}, we know that, there exists $T_{0}\gg 1$ such that 
\begin{equation*}
\forall t\in[T_{0},T_{m}],\quad \bigg\|\vec{u}_{m}(t)-\sum_{n=1}^{N}\vec{Q}_{n}(t)\bigg\|_{\E}\le \epsilon,
\end{equation*} 
where $\epsilon>0$ is the constant of Proposition~\ref{prop:weak}.

From the uniform estimates obtained in~\eqref{uniest} at $T=T_{0}$, up to the extraction of a subsequence,
there exists $\vec{u}_{0}=(u_{0},u_{1})\in \dot{H}^{1}\times L^{2}$ such that 
\begin{equation*}
\vec{u}_{m}(T_{0})\rightharpoonup\vec{u}_{0}\quad \mbox{in}\ \dot{H}^{1}\times L^{2}-weak\quad \mbox{as}\ m\to \infty.
\end{equation*}
Fix $T>T_{0}$. From Proposition~\ref{prop:weak} applied to the compact set 
\begin{equation*}
{\rm{K}}=\left\{\sum_{n=1}^{N}\vec{Q}_{n}(t)\in \dot{H}^{1}\times L^{2}, t\in [T_{0},T]\right\},
\end{equation*}
the solution $\vec{u}(t)$ of~\eqref{wave} corresponding to $\vec{u}(T_{0})=\vec{u}_{0}$ is well-defined and it holds $\vec{u}_{m}(t)\rightharpoonup \vec{u}(t)$ weakly in $\dot{H}^{1}\times L^{2}$ on $[T_{0},T]$. By~\eqref{uniest} and the properties of weak convergence, the solution $\vec{u}$ satisfies, for all $t\in [T_{0},T]$, 
\begin{equation*}
\bigg\|\vec{u}(t)-\sum_{n=1}^{N}\vec{Q}_{n}(t)\bigg\|_{\E}\le \mathop{\underline{\lim}}_{m \to \infty} \bigg\|\vec{u}_{m}(t)-\sum_{n=1}^{N}\vec{Q}_{n}(t)\bigg\|_{\E}\le t^{-\frac{8}{5}}.
\end{equation*} 
Since $T\ge T_{0}$ is arbitrary, the solution $\vec{u}$ is defined and satisfies the conclusion of Theorem~\ref{main:thm} on $[T_{0},\infty)$.
The proof of Theorem~\ref{main:thm} from Proposition~\ref{main:pro} is complete.

\smallskip

The rest of this section is devoted to the proof of Proposition~\ref{main:pro}.
\subsection{Bootstrap setting}\label{SS:Boot}
For $0<t\le T_{m}$, as long as $\vec{u}_{m}(t)$ is well defined in $\dot{H}^{1}\times L^{2}$
and satisfies~\eqref{estQ0}, we decompose $\vec{u}_{m}(t)$ as in Proposition~\ref{le:deco}.
In particular, we denote by $\vec{\varepsilon}=(\varepsilon,\eta)$, $\bb=(b_{n,k})_{(n,k)\in I^{0}}$,
$(a_{n,j}^{\pm})_{(n,j)\in I}$ the parameters of the decomposition of $\vec{u}_{m}$.

\smallskip

To prove Proposition~\ref{main:pro}, we introduce the following bootstrap estimates:
\begin{equation}\label{Bootset}
\begin{aligned}
\|\vec{\varepsilon}(t)\|_{\E}&\le t^{-\frac{29}{10}},\quad \sum_{(n,k)\in I^{0}}|b_{n,k}(t)|\le t^{-\frac{9}{5}},\\
\sum_{(n,j)\in I}|a_{n,j}^{+}(t)|^{2}&\le t^{-6},\quad \ \sum_{(n,j)\in I}|a_{n,j}^{-}(t)|^{2}\le t^{-\frac{59}{10}}.
\end{aligned}
\end{equation}

For $\boldsymbol{a}_{m}\in B_{\R^{|I|}}\big(T_{m}^{-3}\big)$, set 
\begin{equation}\label{defT*}
T_{*}(\boldsymbol{a}_{m})=\inf \{t\in [T_{0},T_{m}];\vec{u}_{m}~\mbox{satisfies}~\eqref{estQ0}~\mbox{and}~\eqref{Bootset}~\mbox{holds on}~[t,T_{m}]\}.
\end{equation}

Note that, for the proof of Proposition~\ref{main:pro}, it suffices to prove that there exists $T_{0}\gg 1$ (independent with $m$) large enough and at least one choice of $\boldsymbol{a}_{m}\in B_{\R^{|I|}}\big(T_{m}^{-3}\big)$
such that $T_{*}(\boldsymbol{a}_{m})=T_{0}$.
\subsection{Energy functional}\label{sec:ener}
Without loss of generality
\begin{equation*}
-1<\ell_{1}<\cdots<\ell_{N}<1.
\end{equation*}
Denote
\begin{equation*}
\ell^{*}=\max_{n}(|\ell_{n}|)<1.
\end{equation*}
For
\begin{equation*}
0<\delta<\frac{1}{100}\min_{n}(\ell_{n+1}-\ell_{n})
\end{equation*}
small enough to be chosen later, we denote
\begin{equation*}
\begin{aligned}
&{\bar{\ell}}_{n}=\ell_{n}+\delta(\ell_{n+1}-\ell_{n}),\quad \mbox{for $n=1,\ldots,N-1$},\\
&\underline{\ell}_{n}=\ell_{n}-\delta(\ell_{n}-\ell_{n-1}),\quad \mbox{for $n=2,\ldots,N$}.
\end{aligned}
\end{equation*}
For $t>0$, denote
\begin{equation*}
\Omega(t)=\left(\big(\bar{\ell}_{1}t,\underline{\ell}_{2}t\big)\cup\ldots\cup\big(\bar{\ell}_{N-1}t,\underline{\ell}_{N}t\big)\right)\times \mathbb{R}^{4},
\quad \Omega^{C}(t)=\mathbb{R}^{5}\setminus \Omega(t).
\end{equation*}
We define the continuous function $\chi_{N}(t,x)=\chi_{N}(t,x_{1})$ as follow (see~\cite{MMwave1,XY} for a similar choice of cut-off function),
for $t>0$,
\begin{equation}\label{defchiN}
\left\{\begin{aligned}
    &\hbox{$\chi_N(t,x) =  \ell_1$ for $x_1\in (-\infty,  \bar{\ell}_1 t]$},\\
	& \hbox{$\chi_N(t,x) = \ell_n $ for $x_1\in [\underline{\ell}_n t, \bar{\ell}_n t]$, for $n\in \{2,\ldots,N-1\}$,}
  	\\ & \hbox{$\chi_N(t,x)= \ell_N$ for $x_1\in [\underline{\ell}_{N} t,+\infty)$},\\
    	& \chi_N(t,x) = \frac{x_1}{(1-2\delta)t}  - \frac {\delta}{1-2 \delta} (\ell_{n+1}+\ell_n) 
	\hbox{ for $x_1 \in [\bar{\ell}_{n} t,\underline{\ell}_{n+1} t ]$, $n\in \{1,\ldots,N-1\}$}.
\end{aligned}
\right.
\end{equation}
Note that 
\begin{equation}\label{derchi}\left\{\begin{aligned}
   & \partial_{x_1} \chi_N(t,x)= \frac{1}{(1-2\delta)t}  \quad \hbox{for $x\in \Omega(t)$},\\
   	& \partial_{t} \chi_N(t,x)= -\frac 1t \frac{x_1}{(1-2\delta)t} \quad \hbox{for $x\in \Omega(t)$},\\
   	& \partial_t \chi_N(t,x) =0,\quad  \nabla \chi_N(t,x)=0, \quad \hbox{for $x\in \Omega^C(t)$}.\\
\end{aligned} \right.\end{equation}

We define (see~\cite{MMwave1,MMT,XY} for similar energy functional)
\begin{equation*}
\mathcal{K}(t)=\mathcal{F}(t)-\mathcal{G}(t)=\mathcal{E}(t)+\mathcal{P}(t)-\mathcal{G}(t)
\end{equation*}
where
\begin{equation*}
\mathcal{E}(t)=\int_{\RR}\left(|\nabla\varepsilon|^{2}+\eta^{2}-2(F(U+V+\varepsilon)-F(U+V)-f(U+V)\varepsilon)\right)\d x,
\end{equation*}
\begin{equation*}
\mathcal{P}(t)=2\int_{\mathbb{R}^{5}}\big(\chi_{N}(t,x)\partial_{x_{1}}\varepsilon(t,x))\eta(t,x)\d x\quad 
\mbox{and}\quad 
\mathcal{G}(t)=2\int_{\RR}\varepsilon(t,x) G_{3}(t,x)\d x.
\end{equation*}
We start with the following technical lemma.
\begin{lemma}\label{le:tech1}
	Let $W$ be a continuous function such that
	\begin{equation}\label{decayW}
	|W(x)|\lesssim \langle x \rangle^{-(4+\alpha)}\quad \mbox{for all}\ x\in \RR,
	\end{equation}
	where $\alpha\ge 0$. For all $n=1,\ldots,N$, the following estimates hold.
	\begin{enumerate}
		\item \emph{Estimate of $L^{\frac{10}{7}}$ norm.}
		\begin{equation}\label{estL710}
		\left\|(\ell_{n}-\chi_{N})W(x-\bell_{n}t)\right\|_{L^{\frac{10}{7}}}\lesssim t^{-(\frac{1}{2}+\alpha)}.
		\end{equation}
		\item \emph{Estimate of $L^{2}$ norm.}
		\begin{equation}\label{estL2}
		\left\|(\ell_{n}-\chi_{N})W(x-\bell_{n}t)\right\|_{L^{2}}\lesssim t^{-(\frac{3}{2}+\alpha)}.
		\end{equation}
		\item \emph{Estimate of $L^{\frac{10}{3}}$ norm.}
		\begin{equation}\label{estL310}
		\left\|(\ell_{n}-\chi_{N})W(x-\bell_{n}t)\right\|_{L^{\frac{10}{3}}}\lesssim t^{-(\frac{5}{2}+\alpha)}.
		\end{equation}
	\end{enumerate}
\end{lemma}
\begin{proof}
	From~\eqref{decayW}, the definition of $\chi_{N}$ and change of variable,
	\begin{equation*}
	\begin{aligned}
	\int_{\RR}|\ell_{n}-\chi_{N}|^{\frac{10}{7}}|W(x-\bell_{n}t)|^{\frac{10}{7}}\d x
	&\lesssim \int_{|x|\ge \delta^{2} t}\langle x\rangle^{-\frac{10}{7}(4+\alpha)}\d x\\
	&\lesssim \int_{\delta^{2} t}^{\infty}\frac{r^{4}}{(1+r^{2})^{\frac{5}{7}(4+\alpha)}}\d r
	\lesssim t^{-\frac{10}{7}(\frac{1}{2}+\alpha)},
	\end{aligned}
	\end{equation*}
	which implies~\eqref{estL710}. The proof of~\eqref{estL2} and~\eqref{estL310} follows from similar arguments and it is omitted.
\end{proof}

Under the bootstrap setting~\eqref{Bootset}, we prove the following estimates.
\begin{proposition}\label{mainpro}
There exists $0<\nu\ll 1$ such that the following hold.
\begin{enumerate}
	\item \emph{Coercivity.}
\begin{equation}\label{energycoer}
\nu \|\vec{\varepsilon}(t)\|_{\E}^{2}\le \mathcal{K}(t)+\nu^{-1}t^{-\frac{59}{10}}.
\end{equation}
 \item \emph{Time variation of $\mathcal{K}$}.
\begin{equation}\label{timevariation}
-\frac{\rm{d}}{{\rm {d}} t}(t^{2}\mathcal{K})(t)\le \nu^{-1}t^{-\frac{49}{10}}.
\end{equation}
\end{enumerate}
\end{proposition}
\begin{proof}
 Proof of (i). We set
\begin{equation*}
\mathcal{F}_{\Omega}(t)=\int_{\Omega}\left(|\nabla \varepsilon|^{2}+\eta^{2}+2\big(\chi_{N}\partial_{x_{1}}\varepsilon\big)\eta\right)\d x,\quad 
\mathcal{F}_{\Omega^{C}}(t)=\int_{\Omega^{C}}\big(|\nabla \varepsilon|^{2}+\eta^{2}\big)\d x.
\end{equation*}
First, from $|\chi_{N}|\le \bar{\ell}<1$,
\begin{equation}\label{est:EOmega}
\begin{aligned}
\mathcal{F}_{\Omega}
&=\bar{\ell}\int_{\Omega}
\left(\frac{\chi_{N}}{\bar{\ell}}\partial_{x_{1}}\varepsilon+\eta\right)^{2}\d x
+\int_{\Omega}\left(\left(1-\frac{\chi_{N}^{2}}{\bar{\ell}}\right)(\partial_{x_{1}}\varepsilon)^{2}
+(1-\bar{\ell})\eta^{2}+|\overline{\nabla}\varepsilon|^{2}\right)\d x\\
&\ge \bar{\ell}\int_{\Omega}
\left(\frac{\chi_{N}}{\bar{\ell}}\partial_{x_{1}}\varepsilon+\eta\right)^{2}\d x
+(1-\bar{\ell})\int_{\Omega}\left(|\nabla\varepsilon|^{2}+\eta^{2}\right)\d x.
\end{aligned}
\end{equation}
Second, by~\eqref{Bootset} and Young's inequality for product, 
\begin{equation}\label{estG}
\big|\mathcal{G}(t)\big|
\lesssim \bigg(\sum_{(n,k)\in I^{0}}|b_{n,k}|^{2}\bigg)\|\vec{\varepsilon}\|_{\E}
\lesssim t^{-\frac{13}{2}}.
\end{equation} 
Therefore the coercivity property for $\mathcal{K}$ is a consequence of the following stronger coercivity property, for some $\alpha>0$,
\begin{equation}\label{coer:E}
\mathcal{F}(t)\ge \mathcal{F}_{\Omega}(t)+\mu \mathcal{F}_{\Omega^{C}}(t)-\mu^{-1}t^{-\frac{59}{10}}
-\mu^{-1}t^{-\alpha}\|\vec{\varepsilon}\|^{2}_{\E}-\mu^{-1}\|\vec{\varepsilon}\|_{\E}^{3}.
\end{equation}

 The coercivity property~\eqref{coer:E} is a standard consequence of the localized coercivity property around one excited solitary wave $q_{\ell}$ in Proposition~\ref{procoer} with the orthogonality relations~\eqref{orth}, and an elementary localization argument.
 We refer to the proof of~\cite[Proposition 4.2 (ii)]{MMwave1} and~\cite[Lemma 5.4 (ii)]{XY} for a similar proof.

 \smallskip
 Proof of (ii). {\bf Step 1.} Time variation of $\mathcal{E}$. We claim
 \begin{equation}\label{estdE}
 \begin{aligned}
 \frac{\d }{\d t}\mathcal{E}
 =&2\int_{\RR}{\rm{Mod}}_{1}\left(\Delta\varepsilon+f'(U)\varepsilon\right)\d x
 -2\int_{\RR}\eta {\rm{Mod}_{2}}\d x+2\int_{\RR}\eta G_{3}\d x\\
 +&2\sum_{n=1}^{N}\int_{\RR}\ell_{n}\partial_{x_{1}}Q_{n}
 \left(f(U+V+\varepsilon)-f(U+V)-f'(U+V)\varepsilon\right)\d x
 +O\big(t^{-\frac{69}{10}}\big).
 \end{aligned}
 \end{equation}
 We decompose
\begin{align*}
\frac{\rm{d}}{{\rm{d}}t}\mathcal{E}=\mathcal{I}_{1}+\mathcal{I}_{2}+\mathcal{I}_{3}+\mathcal{I}_{4},
\end{align*}
where 
\begin{equation*}
\mathcal{I}_{1}=-2\int_{\RR}\partial_{t}U\big(f(U+V+\varepsilon)-f(U+V)-f'(U+V)\varepsilon\big)\d x,
\end{equation*}
\begin{equation*}
\mathcal{I}_{2}=-2\int_{\RR}\partial_{t}V\big(f(U+V+\varepsilon)-f(U+V)-f'(U+V)\varepsilon\big)\d x,
\end{equation*}
\begin{equation*}
\mathcal{I}_{3}=2\int_{\RR}\partial_{t}\varepsilon\big(-\Delta \varepsilon-f(U+V+\varepsilon)+f(U+V)\big)\d x,\quad
\mathcal{I}_{4}=2\int_{\RR}\eta\partial_{t}\eta \d x.
\end{equation*}
\emph{Estimate on $\mathcal{I}_{1}$.} 
We claim
\begin{equation}\label{est:I1}
\mathcal{I}_{1}
=2\sum_{n=1}^{N}\int_{\RR}\ell_{n}\partial_{x_{1}}Q_{n}\big(f(U+V+\varepsilon)-f(U+V)-f'(U+V)\varepsilon\big)\d x.
\end{equation}
By direct computation, we obtain
\begin{equation*}
\partial_{t}U=\sum_{n=1}^{N}\partial_{t} Q_{n}=-\sum_{n=1}^{N}\ell_{n}\partial_{x_{1}}Q_{n},
\end{equation*}
which implies~\eqref{est:I1}.

\emph{Estimate on $\mathcal{I}_{2}$.} We claim
\begin{equation}\label{est:I2}
\left|\mathcal{I}_{2}\right|\lesssim t^{-7}.
\end{equation}
By direct computation,
\begin{equation*}
\partial_{t}V=\sum_{(n,k)\in I^{0}}\dot{b}_{n,k}\Psi_{n,k}-\sum_{(n,k)\in I^{0}}b_{n,k}\ell_{n}\partial_{x_{1}}\Psi_{n,k}.
\end{equation*}
Thus, from~\eqref{Taylor1},~\eqref{estb},~\eqref{Bootset}, and the decay properties of $\Psi_{n,k}$,
\begin{equation*}
\begin{aligned}
\left|\mathcal{I}_{2}\right|
&\lesssim\sum_{(n,k)\in I^{0}} \int_{\RR}\left(|\dot{b}_{n,k}||\Psi_{n,k}|+|b_{n,k}||\ell_{n}\partial_{x_{1}}\Psi_{n,k}|\right)
\left(\big(|U|^{\frac{1}{3}}+|V|^{\frac{1}{3}}\big)|\varepsilon|^{2}+|\varepsilon|^{\frac{7}{3}}\right)\d x\\
&\lesssim \sum_{(n,k)\in I^{0}}\left(\|\vec{\varepsilon}\|_{\E}+|b_{n,k}|+t^{-4}\right)\|\vec{\varepsilon}\|_{\E}^{2}
\lesssim t^{-\frac{87}{10}}+t^{-\frac{38}{5}}+t^{-7}\lesssim t^{-7},
\end{aligned}
\end{equation*}
which means~\eqref{est:I2}.

\emph{Estimate on $\mathcal{I}_{3}$.} We prove the following estimate
\begin{equation}\label{est:I3}
\begin{aligned}
\mathcal{I}_{3}
=&-2\int_{\mathbb{R}^{5}}\eta\big(\Delta \varepsilon+f(U+V+\varepsilon)-f(U+V)\big)\d x\\
&+2\int_{\mathbb{R}^{5}}{\rm{Mod}}_{1}\big(\Delta\varepsilon+f'(U)\varepsilon\big)\d x+O(t^{-7}).
\end{aligned}
\end{equation}
By direct computation and~\eqref{equpe},
\begin{align*}
\mathcal{I}_{3}=
&-2\int_{\mathbb{R}^{5}}\eta\big(\Delta \varepsilon+f(U+V+\varepsilon)-f(U+V)\big)\d x\\
&+2\int_{\mathbb{R}^{5}}{\rm{Mod}}_{1}\big(\Delta\varepsilon+f'(U)\varepsilon\big)\d x
+\mathcal{I}_{3,1}+\mathcal{I}_{3,2},
\end{align*}
where
\begin{equation*}
\begin{aligned}
\mathcal{I}_{3,1}&=2\int_{\R^{5}}{\rm{Mod}}_{1}\left(f'(U+V)-f'(U)\right)\varepsilon\d x,\\
\mathcal{I}_{3,2}&=2\int_{\R^{5}}{\rm{Mod}}_{1}\left(f(U+V+\varepsilon)-f(U+V)-f'(U+V)\varepsilon\right)\d x.
\end{aligned}
\end{equation*}
By the Cauchy-Schwarz inequality,~\eqref{Taylor2},~\eqref{Taylor1},~\eqref{estb} and~\eqref{Bootset},
\begin{equation*}
\begin{aligned}
\left|\mathcal{I}_{3,1}\right|&\lesssim \|{\rm{Mod}}_{1}\|_{L^{\frac{10}{3}}}\|f'(U+V)-f'(U)\|_{L^{\frac{5}{2}}}\|\varepsilon\|_{L^{\frac{10}{3}}}\\
&\lesssim \sum_{(n,k)\in I^{0}}|b_{n,k}| \bigg(\|\vec{\varepsilon}\|_{\E}+\sum_{(n',k')\in I^{0}}|b_{n',k'}|^{2}+t^{-4}\bigg)\|\vec{\varepsilon}\|_{\E}\\
&\lesssim t^{-\frac{9}{5}}\left(t^{-\frac{29}{10}}+t^{-\frac{18}{5}}+t^{-4}\right)t^{-\frac{29}{10}}\lesssim t^{-7},
\end{aligned}
\end{equation*}
\begin{equation*}
\begin{aligned}
\left|\mathcal{I}_{3,2}\right|
&\lesssim \|{\rm{Mod}}_{1}\|_{L^{\frac{10}{3}}}\|f(U+V+\varepsilon)-f(U+V)-f'(U+V)\varepsilon\|_{L^{\frac{10}{7}}}\\
&\lesssim \bigg(\|\vec{\varepsilon}\|_{\E}+\sum_{(n,k)\in I^{0}}|b_{n,k}|^{2}+t^{-4}\bigg)\|\vec{\varepsilon}\|^{2}_{\E}\lesssim  
\left(t^{-\frac{29}{10}}+t^{-\frac{18}{5}}+t^{-4}\right)t^{-\frac{29}{5}}\lesssim t^{-7}.
\end{aligned}
\end{equation*}
We see that~\eqref{est:I3} follows from above estimates.

\emph{Estimate on $\mathcal{I}_{4}$.} We claim
\begin{equation}\label{est:I4}
\begin{aligned}
\mathcal{I}_{4}
&=2\int_{\RR}\eta\big(\Delta \varepsilon+f(U+V+\varepsilon)-f(U+V)\big)\d x
\\
&-2\int_{\RR}\eta {\rm{Mod}}_{2}\d x+2\int_{\RR}\eta G_{3}\d x+O\big(t^{-\frac{69}{10}}\big).
\end{aligned}
\end{equation} 
We compute, using again~\eqref{equpe},
\begin{align*}
{\mathcal{I}_{4}}
=&2\int_{\RR}\eta\big(\Delta\varepsilon+f(U+V+\varepsilon)-f(U+V)\big)\d x\\
-&2\int_{\RR}\eta {\rm{Mod}}_{2}\d x+2\int_{\RR}\eta G_{3}\d x+2\int_{\RR}\eta G_{1}\d x+2\int_{\RR}\eta (G_{2}-G_{3})\d x.
\end{align*}
Indeed, from~\eqref{est:G123},~\eqref{est:G4} and~\eqref{Bootset}, we obtain
\begin{equation*}
\begin{aligned}
\left|\int_{\RR}\eta G_{1}\d x\right|+\bigg|\int_{\RR}\eta (G_{2}-G_{3})\d x\bigg|
&\lesssim \|\eta\|_{ L^{2}}\left(\|G_{1}\|_{L^{2}}+\|G_{2}-G_{3}\|_{L^{2}}\right)\\
&\lesssim \|\eta\|_{L^{2}}\bigg(\sum_{(n,k)\in I^{0}}|b_{n,k}|^{\frac{7}{3}}+t^{-4}\bigg)\lesssim t^{-\frac{69}{10}},
\end{aligned}
\end{equation*}
which implies~\eqref{est:I4}.

In conclusion of estimates~\eqref{est:I1},~\eqref{est:I2},~\eqref{est:I3} and~\eqref{est:I4},
 we obtain~\eqref{estdE}.

{\bf{Step 2.}} Time variation of $\mathcal{P}$. We claim
\begin{equation}\label{estdP}
\begin{aligned}
\frac{\d }{\d t}\mathcal{P}
&=-\frac{1}{(1-2\delta)t}\int_{\Omega}\big(\eta^{2}+(\partial_{x_{1}}\varepsilon)^{2}
+2\frac{x_{1}}{t}(\partial_{x_{1}}\varepsilon)\eta-|\overline{\nabla}\varepsilon|^{2}\big)\d x\\
&-2\sum_{n=1}^{N}\int_{\RR}\chi_{N}\partial_{x_{1}}Q_{n}
\big(f(U+V+\varepsilon)-f(U+V)-f'(U+V)\varepsilon\big)\d x\\
&-2\int_{\RR}\chi_{N}\left(\eta \partial_{x_{1}}{\rm{Mod}_{1}}+{\rm{Mod}}_{2}\partial_{x_{1}}\varepsilon\right)\d x+2\int_{\RR}\chi_{N}(\partial_{x_{1}}\varepsilon)G_{3}\d x +O(t^{-\frac{69}{10}}).
\end{aligned}
\end{equation}
We decompose
\begin{equation}
\begin{aligned}
\frac{\rm{d}}{\rm{d}t}\mathcal{P}
&=2\int_{\RR}(\partial_{t}\chi_{N})(\partial_{x_{1}}\varepsilon)\eta \d x+2\int_{\RR}\chi_{N}\partial_{t}\big((\partial_{x_{1}}\varepsilon)\eta\big)\d x=\mathcal{I}_{5}+\mathcal{I}_{6}.
\end{aligned}
\end{equation}
\emph{Estimate on $\mathcal{I}_{5}$.} From~\eqref{derchi}, we obtain
\begin{equation}\label{estI5}
{\mathcal{I}_{5}}=-\frac{2}{(1-2\delta)t}\int_{\Omega}\frac{x_{1}}{t}(\partial_{x_{1}}\varepsilon)\eta \d x.
\end{equation}

\emph{Estimate on $\mathcal{I}_{6}$.} We claim
\begin{equation}\label{estI6}
\begin{aligned}
{\mathcal{I}_{6}}&=-\frac{1}{(1-2\delta)t}\int_{\Omega}\big(\eta^{2}+(\partial_{x_{1}}\varepsilon)^{2}
-|\overline{\nabla}\varepsilon|^{2}\big)\d x\\
&-2\sum_{n=1}^{N}\int_{\RR}\chi_{N}\partial_{x_{1}}Q_{n}
\big(f(U+V+\varepsilon)-f(U+V)-f'(U+V)\varepsilon\big)\d x\\
&-2\int_{\RR}\chi_{N}\left(\eta \partial_{x_{1}}{\rm{Mod}_{1}}+{\rm{Mod}}_{2}\partial_{x_{1}}\varepsilon\right)\d x+2\int_{\RR}\chi_{N}(\partial_{x_{1}}\varepsilon)G_{3}\d x +O(t^{-\frac{69}{10}}).
\end{aligned}
\end{equation} 
By direct computation, we decompose
\begin{align*}
\mathcal{I}_{6}
=2\int_{\RR}\chi_{N}\left(\partial_{x_{1}}(\partial_{t}\varepsilon)\right)\eta\d x+2\int_{\RR}\chi_{N}(\partial_{x_{1}}\varepsilon)\partial_{t}\eta\d x
=\mathcal{I}_{6,1}+\mathcal{I}_{6,2}.
\end{align*}
From~\eqref{equpe},~\eqref{derchi}, and integration by parts,
\begin{equation}\label{estI61}
\mathcal{I}_{6,1}=-\frac{1}{(1-2\delta)t}\int_{\Omega}\eta^{2}\d x-2\int_{\RR}\chi_{N}\left(\eta \partial_{x_{1}}{\rm{Mod}_{1}}\right)\d x.
\end{equation}
Using again~\eqref{equpe},~\eqref{derchi}, and integration by parts,
\begin{equation}
\begin{aligned}
\mathcal{I}_{6,2}
&=-\frac{1}{(1-2\delta)t}\int_{\Omega}\left((\partial_{x_{1}}\varepsilon)^{2}-
|\overline{\nabla}\varepsilon|^{2}\right)\d x
-2\int_{\RR}\chi_{N}(\partial_{x_{1}}\varepsilon){\rm{Mod}}_{2}\d x\\
&+2\int_{\RR}\chi_{N}\left(\partial_{x_{1}}\varepsilon\right) G_{3}\d x
+2\int_{\RR}\chi_{N}\partial_{x_{1}}\varepsilon\left(f(U+V+\varepsilon)-f(U+V)\right)\d x\\
&+2\int_{\RR}\chi_{N}(\partial_{x_{1}}\varepsilon)G_{1}\d x+2\int_{\RR}\chi_{N}(\partial_{x_{1}}\varepsilon)(G_{2}-G_{3})\d x.
\end{aligned}
\end{equation}
Note that, by integration by parts,
\begin{equation*}
\begin{aligned}
&2\int_{\RR}\chi_{N}\partial_{x_{1}}\varepsilon\left(f(U+V+\varepsilon)-f(U+V)\right)\d x\\
=&-2\sum_{n=1}^{N}\int_{\RR}\chi_{N}\partial_{x_{1}}Q_{n}\big(f(U+V+\varepsilon)-f(U+V)-f'(U+V)\varepsilon\big)\d x+\mathcal{I}_{6}^{1}+\mathcal{I}_{6}^{2},
\end{aligned}
\end{equation*}
where
\begin{equation*}
\mathcal{I}_{6}^{1}=-\frac{2}{(1-2\delta)t}\int_{\Omega}\left(F(U+V+\varepsilon)-F(U+V)-f(U+V)\varepsilon\right)\d x,
\end{equation*}
\begin{equation*}
\mathcal{I}_{6}^{2}=-2\sum_{(n,k)\in I^{0}}b_{n,k}\int_{\RR}\chi_{N}\partial_{x_{1}}\Psi_{n,k}
\left(f(U+V+\varepsilon)-f(U+V)-f'(U+V)\varepsilon\right)\d x.
\end{equation*}
From~\eqref{Taylor4},~\eqref{Bootset}, the decay properties of $Q_{n}$ and $\Psi_{n,k}$ and Sobolev embedding theorem,
\begin{equation*}
	\begin{aligned}
\left|\mathcal{I}^{1}_{6}\right|
&\lesssim t^{-1}\int_{\Omega}\left((|U|^{\frac{4}{3}}+|V|^{\frac{4}{3}})\varepsilon^{2}
+\varepsilon^{\frac{10}{3}}\right)\d x\\
&\lesssim t^{-1}\left(t^{-1}\|\vec{\varepsilon}\|^{2}_{\E}+\|\vec{\varepsilon}\|_{\E}^{\frac{10}{3}}\right)
\lesssim t^{-1}(t^{-\frac{34}{5}}+t^{-\frac{29}{3}})\lesssim t^{-7}.
\end{aligned}
\end{equation*}
From~\eqref{Taylor1},~\eqref{Bootset} and Sobolev embedding theorem,
\begin{equation*}
\begin{aligned}
\left|\mathcal{I}^{2}_{6}\right|
&\lesssim \sum_{(n,k)\in I^{0}}|b_{n,k}|
\int_{\RR}|\partial_{x_{1}}\Psi_{n,k}|\left((|U|^{\frac{1}{3}}+|V|^{\frac{1}{3}})|\varepsilon|^{2}
+|\varepsilon|^{\frac{7}{3}}\right)\d x\\
&\lesssim \sum_{(n,k)\in I^{0}}|b_{n,k}|
\left(\|\vec{\varepsilon}\|_{\E}^{2}+\|\vec{\varepsilon}\|_{\E}^{\frac{7}{3}}\right)
\lesssim t^{-\frac{9}{5}}(t^{-\frac{29}{5}}+t^{-\frac{20}{3}})
\lesssim t^{-\frac{38}{5}}.
\end{aligned}
\end{equation*}
Next, by the Cauchy-Schwarz inequality,~\eqref{est:G123},~\eqref{est:G4} and~\eqref{Bootset},
\begin{equation*}
\begin{aligned}
&\left|\int_{\RR}\chi_{N}(\partial_{x_{1}}\varepsilon)G_{1}\d x\right|+\left|\int_{\RR}\chi_{N}(\partial_{x_{1}}\varepsilon)(G_{2}-G_{3})\d x\right|\\
&\lesssim
\|\partial_{x_{1}}\varepsilon\|_{L^{2}}\left(\|G_{1}\|_{L^{2}}+\|G_{2}-G_{3}\|_{L^{2}}\right)\\
&\lesssim
\|\partial_{x_{1}}\varepsilon\|_{L^{2}}\bigg(\sum_{(n,k)\in  I^{0}}|b_{n,k}|^{\frac{7}{3}}+t^{-4}\bigg)
\lesssim t^{-\frac{29}{10}}\left(t^{-\frac{21}{5}}+t^{-4}\right)
\lesssim t^{-\frac{69}{10}}.
\end{aligned}
\end{equation*}
Gathering above estimates, we obtain~\eqref{estI6}.

In conclusion of estimates~\eqref{estI5} and~\eqref{estI6}, we obtain~\eqref{estdP}.

 {\bf Step 3.} Time variation of $\mathcal{G}$. We claim
 \begin{equation}\label{estdG}
 \frac{\d}{\d t}\mathcal{G}=2\int_{\RR}\eta G_{3}\d x+2\sum_{n=1}^{N}\int_{\RR}\ell_{n}(\partial_{x_{1}}\varepsilon)G_{3,n}\d x+O(t^{-7}).
 \end{equation}
 By direct computation and integration by parts,
 \begin{equation*}
 \frac{\d}{\d t}\mathcal{G}=2\int_{\RR}\eta G_{3}\d x+2\sum_{n=1}^{N}\int_{\RR}\ell_{n}(\partial_{x_{1}}\varepsilon)G_{3,n}\d x
 +\mathcal{I}_{7}+\mathcal{I}_{8}+\mathcal{I}_{9},
 \end{equation*}
 where
  \begin{equation*}
 \mathcal{I}_{7}=\sum_{n=1}^{N}\sum_{k,k'=1}^{K_{n}}\dot{b}_{n,k}b_{n,k'}\int_{\RR}\varepsilon\big(f''(Q_{n})\Psi_{n,k}\Psi_{n,k'}\big)\d x,
 \end{equation*}
 \begin{equation*}
 \mathcal{I}_{8}=-\sum_{n=1}^{N}\sum_{k=1}^{K_{n}}\dot{b}_{n,k}\int_{\RR}\Psi_{n,k}G_{3,n}\d x\quad \mbox{and}\quad 
 \mathcal{I}_{9}=-\sum_{n\ne n'}\sum_{k=1}^{K_{n}}\dot{b}_{n,k}\int_{\RR}\Psi_{n,k}G_{3,n'}\d x.
 \end{equation*}
 First, from~\eqref{estb},~\eqref{Bootset} and Sobolev embedding,
 \begin{equation*}
 \begin{aligned}
 \left|\mathcal{I}_{7}\right|
 &\lesssim \sum_{n=1}^{N}\sum_{k,k'=1}^{K_{n}}\left|\dot{b}_{n,k}b_{n,k'}\right|
 \int_{\RR}\left|\varepsilon\right|\left|f''(Q_{n})\Psi_{n,k}\Psi_{n,k'}\right|\d x\\
 &\lesssim \sum_{(n,k)\in I^{0}}|b_{n,k}|\|\vec{\varepsilon}\|_{\E}
 \bigg(\|\vec{\varepsilon}\|_{\E}+\sum_{(n,k)\in  I^{0}}b_{n,k}^{2}+t^{-4}\bigg)\\
 &\lesssim t^{-\frac{9}{5}}\left(t^{-\frac{29}{10}}+t^{-\frac{18}{5}}+t^{-4}\right)t^{-\frac{29}{10}}\lesssim t^{-7}.
 \end{aligned}
 \end{equation*}
 Next, from (iii) of Lemma~\ref{le:Q}, we have for all $(n,k)\in I^{0}$,
 \begin{equation*}
 \begin{aligned}
 \int_{\RR}\Psi_{n,k}G_{3,n}\d x
 &=\frac{1}{2}\sum_{k',k''=1}^{K_{n}}b_{n,k'}b_{n,k''}\int_{\RR}\Psi_{n,k}
 \left(f''(Q_{n})\Psi_{n,k'}\Psi_{n,k''}\right)\d x\\
 &=\frac{1}{2}\sum_{k',k''=1}^{K}b_{n,k'}b_{n,k''}(1-\ell_{n}^{2})^{\frac{1}{2}}\int_{\RR}f''(Q)\left(\Psi_{k}\Psi_{k'}\Psi_{k''}\right)\d x=0.
 \end{aligned}
 \end{equation*}
 It follows that $\mathcal{I}_{8}=0$. Last, from Lemma~\ref{le:tech},~\eqref{estb} and~\eqref{Bootset},
 \begin{equation*}
 \begin{aligned}
\left|\mathcal{I}_{9}\right|
&\lesssim \sum_{n'\ne n}\sum_{k=1}^{K_{n}}\sum_{k',k''=1}^{K_{n'}}
\big|\dot{b}_{n,k}b_{n',k'}b_{n',k''}\big|\int_{\RR}\left|\Psi_{n,k}
\left(f''(Q_{n'})\Psi_{n',k'}\Psi_{n',k''}\right)\right|\d x\\
&\lesssim t^{-1}\sum_{(n,k)\in I^{0}}b_{n,k}^{2}\bigg(\|\vec{\varepsilon}\|_{\E}+\sum_{(n,k)\in  I^{0}}b_{n,k}^{2}+t^{-4}\bigg)\\
&\lesssim t^{-1}\left(t^{-\frac{29}{10}}+t^{-\frac{18}{5}}+t^{-4}\right)t^{-\frac{18}{5}}\lesssim t^{-7}.
\end{aligned}
 \end{equation*}
 Gathering above estimates, we obtain~\eqref{estdG}.

{\bf Step 4.} Conclusion. Combining estimates~\eqref{estdE},~\eqref{estdP} and~\eqref{estdG}, we obtain
\begin{equation*}
\frac{\d }{\d t}\mathcal{K}=\mathcal{J}_{1}+\mathcal{J}_{2}+\mathcal{J}_{3}+\mathcal{J}_{4}+\mathcal{J}_{5}
+O\big(t^{-\frac{69}{10}}\big).
\end{equation*}
where
\begin{equation*}
\mathcal{J}_{1}=-\frac{1}{(1-2\delta)t}\int_{\Omega}\left(\eta^{2}+(\partial_{x_{1}}\varepsilon)^{2}
+2\frac{x_{1}}{t}(\partial_{x_{1}}\varepsilon)\eta-|\overline{\nabla}\varepsilon|^{2}\right)\d x,
\end{equation*}
\begin{equation*}
\mathcal{J}_{2}=2\int_{\RR}\left(\Delta\varepsilon+f'(U)\varepsilon\right){\rm{Mod}}_{1}\d x-2\int_{\RR}\chi_{N}(\partial_{x_{1}}\varepsilon){\rm{Mod}}_{2}\d x,
\end{equation*}
\begin{equation*}
\mathcal{J}_{3}=2\sum_{n=1}^{N}\int_{\RR}\left(\ell_{n}-\chi_{N}\right)\partial_{x_{1}}Q_{n}
\left(f(U+V+\varepsilon)-f(U+V)-f'(U+V)\varepsilon\right)\d x,
\end{equation*}
\begin{equation*}
\mathcal{J}_{4}=-2\int_{\RR}\eta \left({\rm{Mod}}_{2}+\chi_{N}\partial_{x_{1}}{\rm{Mod}}_{1}\right)\d x,\ \
\mathcal{J}_{5}=2\sum_{n=1}^{N}\int_{\RR}\left(\chi_{N}-\ell_{n}\right)(\partial_{x_{1}}\varepsilon)G_{3,n}\d x.
\end{equation*}
\emph{Estimate on $\mathcal{J}_{1}$.} For $0<\delta\ll 1$ small enough and $T_{0}$ large enough, we claim
\begin{equation}\label{est:J1}
-\mathcal{J}_{1}\le 2t^{-1}\mathcal{K}+O\big(t^{-\frac{69}{10}}\big).
\end{equation}
From~\eqref{energycoer},~\eqref{est:EOmega},~\eqref{estG} and~\eqref{coer:E}, we have
\begin{equation*}
\begin{aligned}
-(1-2\delta)t\mathcal{J}_{1}
=&\bar{\ell}\int_{\Omega}
\left(\frac{\chi_{N}}{\bar{\ell}}\partial_{x_{1}}\varepsilon+\eta\right)^{2}\d x
+\int_{\Omega}\left(\left(1-\frac{\chi_{N}^{2}}{\bar{\ell}}\right)(\partial_{x_{1}}\varepsilon)^{2}
-|\overline{\nabla}\varepsilon|^{2}\right)\d x\\
&+(1-\bar{\ell})\int_{\Omega}\eta^{2}\d x+2\int_{\Omega}\left(\frac{x_{1}}{t}-\chi_{N}\right)(\partial_{x_{1}}\varepsilon)\eta \d x\\
\le& \mathcal{F}_{\Omega}+O\left(\delta\|\vec{\varepsilon}\|_{\E (\Omega)}^{2}\right)
\le (1+O(\delta+T^{-\alpha}_{0}))\mathcal{K}+O\big(t^{-\frac{59}{10}}\big).
\end{aligned}
\end{equation*}
Letting $\delta$ small enough and $T_{0}$ large enough in above estimate, we obtain~\eqref{est:J1}.

\emph{Estimate on $\mathcal{J}_{2}$.}
By integration by parts,~\eqref{derchi} and 
$-(1-\ell_{n}^{2})\partial_{x_{1}}^{2}\Psi_{n,k}-\bar{\Delta}\Psi_{n,k}-f'(Q_{n})\Psi_{n,k}=0$ for all $(n,k)\in I^{0}$, we have
\begin{equation*}
\mathcal{J}_{2}=\mathcal{J}_{2,1}+\mathcal{J}_{2,2}+\mathcal{J}_{2,3},
\end{equation*}
where
\begin{equation*}
	\begin{aligned}
\mathcal{J}_{2,1}&=\frac{2}{(1-2\delta)t}\int_{\Omega}\varepsilon{\rm{Mod}}_{2}\d x,\\
\mathcal{J}_{2,2}&=2\sum_{(n,k)\in I^{0}}\dot{b}_{n,k}
\int_{\RR}\varepsilon\left(\ell_{n}-\chi_{N}\right)\ell_{n}\partial_{x_{1}}^{2}\Psi_{n,k}\d x,\\
\mathcal{J}_{2,3}&=2\sum_{(n,k)\in I^{0}}\dot{b}_{n,k}\int_{\RR}\varepsilon
\left(f'(U)-f'(Q_{n})\right)\Psi_{n,k}\d x.
\end{aligned}
\end{equation*}
From the Cauchy-Schwarz inequality,~\eqref{estb},~\eqref{Bootset},~\eqref{estL710} and the decay properties of $\Psi_{n,k}$,
\begin{equation*}
\begin{aligned}
\left|\mathcal{J}_{2,1}\right|
&\lesssim t^{-1}\|\varepsilon\|_{L^{\frac{10}{3}}}\sum_{(n,k)\in I^{0}}|\dot{b}_{n,k}|\|\partial_{x_{1}}\Psi_{n,k}\|_{L^{\frac{10}{7}}(\Omega)}\\
&\lesssim t^{-\frac{3}{2}}\|\vec{\varepsilon}\|_{\E}
\bigg(\|\vec{\varepsilon}\|_{\E}+\sum_{(n,k)\in I^{0}}|b_{n,k}|^{2}+t^{-4}\bigg)\lesssim t^{-7},
\end{aligned}
\end{equation*}
\begin{equation*}
\begin{aligned}
\left|\mathcal{J}_{2,2}\right|&\lesssim \|\varepsilon\|_{L^{\frac{10}{3}}}\sum_{(n,k)\in I^{0}}|\dot{b}_{n,k}|\left\|(\ell_{n}-\chi_{N})\partial^{2}_{x_{1}}\Psi_{n,k}\right\|_{L^{\frac{10}{7}}}\\
&\lesssim t^{-\frac{3}{2}}\|\vec{\varepsilon}\|_{\E}
\bigg(\|\vec{\varepsilon}\|_{\E}+\sum_{(n,k)\in I^{0}}|b_{n,k}|^{2}+t^{-4}\bigg)\lesssim t^{-7}.
\end{aligned}
\end{equation*}
From (i) of Lemma~\ref{le:tech} and~\eqref{Taylor1},
\begin{equation*}
\begin{aligned}
&\left\|\left(f'(U)-f'(Q_{n})\right)\Psi_{n,k}\right\|_{L^{\frac{10}{7}}}\\
&\lesssim \sum_{n'\ne n}\left\||Q_{n}|^{\frac{1}{3}}|Q_{n'}|
\Psi_{n,k}\right\|_{L^{\frac{10}{7}}}+
\sum_{n'\ne n}\left\||Q_{n'}|^{\frac{4}{3}}
\Psi_{n,k}\right\|_{L^{\frac{10}{7}}}\lesssim t^{-2}.
\end{aligned}
\end{equation*}
Thus, by~\eqref{estb},~\eqref{Bootset} and Cauchy-Schwarz inequality,
\begin{equation*}
\begin{aligned}
\left|\mathcal{J}_{2,3}\right|
&\lesssim 
\|\varepsilon\|_{L^{\frac{10}{3}}}\sum_{(n,k)\in I^{0}}|\dot{b}_{n,k}|\left\|\left(f'(U)-f'(Q_{n})\right)\Psi_{n,k}\right\|_{L^{\frac{10}{7}}}\\
 &\lesssim t^{-2}\|\vec{\varepsilon}\|_{\E}
\bigg(\|\vec{\varepsilon}\|_{\E}+\sum_{(n,k)\in I^{0}}|b_{n,k}|^{2}+t^{-4}\bigg)\lesssim t^{-7}.
\end{aligned}
\end{equation*}
Gathering above estimates, we obtain
\begin{equation}\label{est:J2}
\left|\mathcal{J}_{2}\right|
\lesssim \left|\mathcal{J}_{2,1}\right|+\left|\mathcal{J}_{2,2}\right|+\left|\mathcal{J}_{2,3}\right|\lesssim t^{-7}.
\end{equation}

\emph{Estimate on $\mathcal{J}_{3}$.} From~\eqref{Taylor1},~\eqref{Bootset},~\eqref{estL310}, Cauchy-Schwarz inequality and the decay properties of $Q_{n}$,
\begin{equation}\label{est:J3}
\begin{aligned}
\left|\mathcal{J}_{3}\right|
&\lesssim 
\sum_{n=1}^{N}\|(\ell_{n}-\chi_{N})\partial_{x_{1}}Q_{n}\|_{L^{\frac{10}{3}}}
\left\|\left(|U|^{\frac{1}{3}}+|V|^{\frac{1}{3}}\right)|\varepsilon|^{2}+|\varepsilon|^{\frac{7}{3}}\right\|_{L^{\frac{10}{7}}}\\
&\lesssim
\|\vec{\varepsilon}\|_{\E}^{2}\sum_{n=1}^{N}\|(\ell_{n}-\chi_{N})\partial_{x_{1}}Q_{n}\|_{L^{\frac{10}{3}}}
\lesssim t^{-7}.
\end{aligned}
\end{equation}

\emph{Estimate on $\mathcal{J}_{4}$.} Note that, 
\begin{equation*}
{\rm{Mod}}_{2}+\chi_{N}\partial_{x_{1}}{\rm{Mod}}_{1}
=-\sum_{(n,k)\in I^{0}}\dot{b}_{n,k}\left(\ell_{n}-\chi_{N}\right)\partial_{x_{1}}\Psi_{n,k}.
\end{equation*}
Therefore, by Cauchy-Schwarz inequality,~\eqref{estb},~\eqref{Bootset} and~\eqref{estL2},
\begin{equation}\label{est:J4}
\begin{aligned}
\left|\mathcal{J}_{4}\right|
&\lesssim \|\eta\|_{L^{2}}\sum_{(n,k)\in I^{0}}|\dot{b}_{n,k}|\|\left(\ell_{n}-\chi_{N}\right)\partial_{x_{1}}\Psi_{n,k}\|_{L^{2}}\\
&\lesssim t^{-\frac{3}{2}}\|\vec{\varepsilon}\|_{\E}
\bigg(\|\vec{\varepsilon}\|_{\E}+\sum_{(n,k)\in I^{0}}|b_{n,k}|^{2}+t^{-4}\bigg)\lesssim t^{-7}.
\end{aligned}
\end{equation}
\emph{Estimate on $\mathcal{J}_{5}$.} From~\eqref{Bootset},~\eqref{estL2} and Cauchy-Schwarz inequality,
\begin{equation}\label{est:J5}
\left|\mathcal{J}_{5}\right|
\lesssim \|\partial_{x_{1}}\varepsilon\|_{L^{2}}\sum_{(n,k)\in I^{0}}|b_{n,k}|^{2}
\|(\ell_{n}-\chi_{N})f''(Q_{n})\Psi_{n,k}^{2}\|_{L^{2}}\lesssim t^{-7}.
\end{equation}

In conclusion of estimates~\eqref{est:J1},~\eqref{est:J2},~\eqref{est:J3},~\eqref{est:J4} and~\eqref{est:J5}, for $\delta$ small enough and $T_{0}$ large enough, we obtain
\begin{equation}
-\frac{\rm{d}}{\rm{d}t}\mathcal{K}\le 2t^{-1}\mathcal{K}+O\big(t^{-\frac{69}{10}}\big),
\end{equation}
which implies~\eqref{timevariation}.
\end{proof}
\subsection{End of the proof of Proposition~\ref{main:pro}}\label{sec:pro}
We start by improving all the estimates in~\eqref{Bootset} except the ones for the unstable directions $(a_{n,j}^{+})_{(n,j)\in I}$.
\begin{lemma}[Closing estimates except $(a_{n,j}^{+})_{(n,j)\in I}$]\label{le:close}
	For all $t\in [T_{*},T_{m}]$,
	\begin{equation}\label{closba}
	\|\vec{\varepsilon}(t)\|_{\E}\le \frac{1}{2}t^{-\frac{29}{10}},\quad \sum_{(n,k)\in I^{0}}|b_{n,k}(t)|\le \frac{1}{2}t^{-\frac{9}{5}} ,\quad 
	\sum_{(n,j)\in I}|a_{n,j}^{-}(t)|^{2}\le \frac{1}{2}t^{-\frac{59}{10}}.
	\end{equation}
\end{lemma}
\begin{proof}
	{\bf{Step 1.}} Bound on the energy norm. First, from~\eqref{est:ab},~\eqref{def:eTm} and $\boldsymbol{a}_{m}\in B_{\R^{|I|}}\big(T_{m}^{-3}\big)$, we have
	\begin{equation*}
	\left|\mathcal{K}(T_{m})\right|\lesssim \left|\mathcal{E}(T_{m})\right|
	+\left|\mathcal{P}(T_{m})\right|+\left|\mathcal{G}(T_{m})\right|\lesssim \|\vec{\varepsilon}(T_{m})\|^{2}_{\E}\lesssim T_{m}^{-6}.
	\end{equation*}
	Thus, integrating~\eqref{timevariation} on $[t, T_{m}]$ for all $t\in [T_{*},T_{m}]$, 
	\begin{equation*}
	\mathcal{K}(t)\le (t/T_{m})^{-2}\mathcal{K}(T_{m})+\nu^{-1}t^{-\frac{59}{10}}
	\lesssim t^{-\frac{59}{10}}.
	\end{equation*}
	Using~\eqref{energycoer}, we have
	\begin{equation*}
	\|\vec{\varepsilon}(t)\|_{\E}^{2}\lesssim \mathcal{K}(t)+t^{-\frac{59}{10}}
	\lesssim t^{-\frac{59}{10}}.
	\end{equation*}
	Letting $T_{0}$ large enough, we conclude~\eqref{closba} on energy norm.
	
	{\bf{Step 2.}} Parameter estimates. Using again~\eqref{est:ab},~\eqref{def:eTm} and $\boldsymbol{a}_{m}\in B_{\R^{|I|}}\big(T_{m}^{-3}\big)$,
	\begin{equation}\label{iniab}
	\sum_{(n,k)\in I^{0}}|b_{n,k}(T_{m})|+\sum_{(n,j)\in I}|a_{n,j}^{-}(T_{m})|\lesssim T_{m}^{-3}.
	\end{equation}
	Integrating~\eqref{estb} on $[t, T_{m}]$ for all $t\in [T_{*},T_{m}]$ and using~\eqref{Bootset},~\eqref{iniab}, we obtain
	\begin{equation*}
	\sum_{(n,k)\in I^{0}}|b_{n,k}(t)|\lesssim \sum_{(n,k)\in I^{0}}|b_{n,k}(T_{m})|
	+\int_{t}^{T_{m}}s^{-\frac{29}{10}}\d s\lesssim t^{-\frac{19}{10}}.
	\end{equation*}
	Letting $T_{0}$ large enough, we conclude~\eqref{closba} for $(b_{n,k})_{(n,k)\in I^{0}}$.
	
	Now we prove the bound on $(a_{n,j}^{-}(t))_{(n,j)\in I}$. By direct computation,~\eqref{equa+-} and~\eqref{Bootset},
	\begin{equation*}
	\begin{aligned}
	\frac{\d }{\d t}\left(e^{-2\alpha_{n,j}t}(a_{n,j}^{-}(t))^{2}\right)
	&=e^{-2\alpha_{n,j}t}a_{n,j}^{-}(t)\left(\frac{\d}{\d t}a_{n,j}^{-}-\alpha_{n,j}a_{n,j}^{-}\right)\\
	&=e^{-2\alpha_{n,j}t}O\bigg(|a_{n,j}^{-}|\big(\|\vec{\varepsilon}\|_{\E}^{2}
	+\sum_{(n,k)\in I^{0}}|b_{n,k}|^{2}+t^{-4}\big)\bigg)\\
	&=e^{-2\alpha_{n,j}t}O\left(t^{-\frac{29}{10}}\left(t^{-\frac{29}{5}}+t^{-\frac{18}{5}}+t^{-4}\right)\right)=e^{-2\alpha_{n,j}t}O\left(t^{-6}\right).
	\end{aligned}
	\end{equation*}
	Integrating on $[t,T_{m}]$ for all $t\in [T_{*},T_{m}]$ and using~\eqref{iniab}, we obtain
	\begin{equation*}
	(a_{n,j}^{-}(t))^{2}\lesssim e^{-2\alpha_{n,j}(T_{m}-t)}(a_{n,j}^{-}(T_{m}))^{2}
	+\int_{t}^{T_{m}} e^{2\alpha_{n,j}(t-s)}s^{-6}\d s\lesssim t^{-6},
	\end{equation*}
	which implies the estimates on $(a_{n,j}^{-})_{(n,j)\in I}$ in~\eqref{closba} for taking $T_{0}$ large enough.
	\end{proof}
Last, we prove the existence of suitable parameters $\boldsymbol{a}_{m}=(a^{m}_{n,j})_{(n,j)\in I}$ by
contradiction, using a topological argument.
\begin{lemma}[Control of unstable directions]\label{le:uni}
	There exist $\boldsymbol{a}_{m}=(a_{n,j}^{m})\in B_{\R^{|I|}}(T^{-3}_{m})$ such that, for $T_{0}$ large enough, $T_{*}(\boldsymbol{a}_{m})=T_{0}$. In particular, the solution $\vec{u}_{m}$ with the initial data $\vec{u}_{m}(T_{m})$ given by Lemma~\ref{le:ini} satisfies~\eqref{uniest}.
\end{lemma}
Note that Lemma~\ref{le:uni} completes the proof of Proposition~\ref{main:pro}.
\begin{proof}
	Let 
	\begin{equation*}
	a(t)=\sum_{(n,j)\in I}(a_{n,j}^{+}(t))^{2}\quad \mbox{and}\quad 
	\bar{\alpha}=\min_{(n,j)\in I} \alpha_{n,j}>0.
	\end{equation*}
	We claim the following transversality property, for any $t\in [T_{*},T_{m}]$ where it holds $a(t)=t^{-6}$,
	we have
	\begin{equation}\label{trans}
	\frac{\d}{\d t}(t^{6}a(t))\le -\bar{\alpha}.
	\end{equation}
	Indeed, from~\eqref{equa+-} and~\eqref{Bootset}, for any $t\in [T_{*},T_{m}]$ where it holds $a(t)=t^{-6}$, we obtain
	\begin{equation*}
	\begin{aligned}
	\frac{\d}{\d t}a(t)&=
	-2\sum_{(n,j)\in I}\alpha_{n,j}(a_{n,j}^{+}(t))^{2}
	+O\bigg(\|\vec{\varepsilon}\|_{\E}\bigg(\|\vec{\varepsilon}\|_{\E}^{2}+\sum_{(n,k)\in I^{0}}b_{n,k}^{2}+t^{-4}\bigg)\bigg)\\
	&=-2\sum_{(n,j)\in I}\alpha_{n,j}(a_{n,j}^{+}(t))^{2}+O\left(t^{-\frac{29}{10}}\left(t^{-\frac{29}{5}}+t^{-\frac{18}{5}}+t^{-4}\right)\right)\\
	&\le -2\bar{\alpha}t^{-6}+O(t^{-\frac{13}{2}})\le -\frac{3}{2}\bar{\alpha}t^{-6}.
	\end{aligned}
	\end{equation*}
	which implies~\eqref{trans} for $T_{0}$ large enough. This transversality property is enough to
	justify the existence of at least a couple $\boldsymbol{a}_{m}=(a_{n,j}^{m})\in B_{\R^{|I|}}(T^{-3}_{m})$
	such that $T_{*}(\boldsymbol{a}_{m})=T_{0}$.
	
	The proof is by contradiction, we assume that for any $\boldsymbol{a}_{m}=(a_{n,j}^{m})\in B_{\R^{|I|}}(T^{-3}_{m})$, it holds $T_{0}<T_{*}(\boldsymbol{a}_{m})\le T_{m}$. Then, a construction follows from the following discussion (see for instance	more details in~\cite{CMM,CMkg} and~\cite[Lemma 4.2]{MMwave1}).
	
	\emph{Continuity of $T_{*}$.} The above transversality property~\eqref{trans} implies that the map
	\begin{equation*}
	\boldsymbol{a}_{m}\in B_{\R^{|I|}}(T^{-3}_{m})\mapsto T_{*}(\boldsymbol{a}_{m})\in (T_{0},T_{m}]
	\end{equation*}
	is continuous and
	\begin{equation*}
	T_{*}(\boldsymbol{a}_{m})=T_{m}\quad \mbox{for}\quad  \boldsymbol{a}_{m}\in S_{\R^{|I|}}(T^{-3}_{m}).
	\end{equation*}
	
	\emph{Construction of a retraction.} We define
	\begin{equation*}
	\begin{aligned}
	\Gamma :\bar{B}_{\R^{|I|}}(T^{-3}_{m})&\mapsto S_{\R^{|I|}}(T^{-3}_{m})\\
	\boldsymbol{a}_{m}&\mapsto \left(\frac{T_{*}(\boldsymbol{a}_{m})}{T_{m}}\right)^{3}(a^{+}_{n,j}(T_{*}(\boldsymbol{a}_{m})))_{(n,j)\in I}.
	\end{aligned}
	\end{equation*}
	From what precedes, $\Gamma$ is continuous. Moreover, $\Gamma$ restricted to $S_{\R^{|I|}}(T^{-3}_{m})$
	 is the identity. The existence of such a map is contradictory with the no retraction theorem for continuous maps from the ball to the sphere. Therefore, the existence of $\boldsymbol{a}_{m}\in B_{\R^{|I|}}(T^{-3}_{m})$ have proved. Then, the uniform estimates~\eqref{uniest} is a consequence of Lemma~\ref{le:close}. The proof of Lemma~\ref{le:uni} is complete.
	\end{proof}

\end{document}